\documentclass[12pt,twoside]{amsart}
\usepackage{amssymb,amsmath,amsfonts,amsthm}
\usepackage{mathrsfs}
\usepackage[X2,T1]{fontenc}
\usepackage[applemac]{inputenc}
\usepackage[mathscr]{euscript}
\usepackage{cite}
\usepackage{latexsym}
\usepackage{verbatim}
\usepackage{cancel}
\usepackage[dvipsnames]{xcolor}
\def\Im{\mathop{\rm Im}\nolimits}

\newcommand{\japx}{\langle x\rangle}
\newcommand{\japxi}{\langle\xi\rangle}

\newcommand{\partialx}{\partial_x}
\newcommand{\partialxi}{\partial_\xi}
\newcommand{\doublepartial}{\partial_\xi^\alpha\partial_x^\beta}

\def\Im{\mathop{\rm Im}\nolimits}

\def\R{\mathbb R}

\def\N{\mathbb N}

\newcommand\dslash{d\llap {\raisebox{.9ex}{$\scriptstyle-\!$}}}
\usepackage{colortbl}

\newcommand{\beqsn}{\arraycolsep1.5pt\begin{eqnarray*}}
\newcommand{\eeqsn}{\end{eqnarray*}\arraycolsep5pt}
\newcommand{\beqs}{\arraycolsep1.5pt\begin{eqnarray}}
\newcommand{\eeqs}{\end{eqnarray}\arraycolsep5pt}

\newtheorem{theorem}{Theorem}
\newtheorem{lemma}{Lemma}

\newtheorem{proposition}{Proposition}
\newtheorem{definition}{Definition}

\newtheorem{remark}{Remark}

\catcode`\@=11

\renewcommand{\section}%
   {\setcounter{equation}{0}\@startsection {section}{1}{\z@}{-3.5ex plus -1ex
  minus -.2ex}{2.3ex plus .2ex}{\Large\bf}}

\title[Cauchy problem for $p$-evolutions equations with data in Gelfand-Shilov spaces]{The Cauchy problem for $p$-evolution equations with initial data in Gelfand-Shilov spaces}

\author[M. Cappiello]{Marco Cappiello $^1$}
\address{Dipartimento di Matematica ``G. Peano'' \\Universit\`a di Torino\\
Via Carlo Alberto 10\\
10123 Torino\\
Italy}
\email{marco.cappiello@unito.it}

\author[E.C. Machado]{Eliakim Cleyton Machado}
\address{Department of Mathematics \\ Federal University of Paran\'a, Caixa Postal 19081 \\ CEP 81531-980, Curitiba, Brazil}
\email{eliakimmachado@gmail.com}

\thanks{The first author has been supported by the Italian Ministry of the University and Research - MUR, within the PRIN 2022 Call (Project Code 2022HCLAZ8, CUP D53C24003370006). The second author wishes to thank for the financial support granted by Conselho de Desenvolvimento Cientifico e Tecnol\'ogico (CNPq), Brazil, 200229/2023-0, during his sandwich Ph.D. period in Turin, Italy, when part of this paper has been written.}

\begin{document}

\def\thefootnote{}
\footnote{$^1$ Corresponding author} 

\begin{abstract} We study the Cauchy problem for a class of linear evolution equations of arbitrary order with coefficients depending both on time and space variables. Under suitable decay assumptions on the coefficients of the lower order terms for $|x|$ large, we prove a well-posedness result in Gelfand-Shilov spaces.
\end{abstract}

\maketitle



\section{Introduction}
The aim of this paper is to study the Cauchy problem with initial data in Gelfand-Shilov spaces for a class of evolution operators of the form \begin{equation}\label{differential_p_evolution_operator}
	P(t,x,D_t,D_x) = D_t + a_p(t)D_x^p + \sum_{j=1}^p a_{p-j}(t,x)D_x^{p-j}, \quad (t,x) \in [0,T] \times \R,
\end{equation}
where $a_p \in C([0,T]; \R)$ and $a_{p-j} \in C([0,T]; \mathcal{B}^\infty(\R)), j=1,\ldots,p.$ Here $D=-i\partial$ and $\mathcal{B}^\infty(\R)$ denotes the space of all smooth complex-valued functions which are uniformly bounded on $\R$ with all their derivatives. Operators of the form \eqref{differential_p_evolution_operator} belong to the class of $p$-evolution operators introduced by Mizohata, cf. \cite{Mizo}. The assumption that $a_p$ is real-valued implies that the principal symbol in the sense of Petrowski admits the real root $\tau =-a_p(t)\xi^p$. Several operators of physical interest are included in the class above, in primis Schr\"odinger operators with lower order terms in the case $p=2$. For $p\geq 3$, operators of the form \eqref{differential_p_evolution_operator} include linearizations of KdV-type operators with variable coefficients ($p=3$) and other models appearing in the theory of dispersive equations like Kawahara-type operators ($p=5$), cf. \cite{Kawahara, Marchenko}.  
The well-posedness of the Cauchy problem \begin{equation}\label{cauchy_problem_gelfand-shilov}
	\left\lbrace \begin{array}{l}
		P(t,x,D_t,D_x) u(t,x) = f(t,x), \quad (t,x) \in [0,T] \times \R \\ 
		u(0,x) = g(x), \quad x \in \R
	\end{array}  \right. 
\end{equation}
has been studied in many papers, in particular, when the data $f,g$ belong to $L^2(\R)$, Sobolev spaces $H^s(\R)$ or Gevrey classes, see \cite{AAC3evolGevrey, AACpevolGevreynec, AACM25, ABZsuff, ABZnec, ACR, CC, CRJEECT, Ichinose1, Ichinose2, KB}. 
The most challenging case is when some of the coefficients $a_{p-j},j=1,\ldots,p$, of the lower order terms are smooth complex-valued functions.
In this case, the well-posedness of the Cauchy problem \eqref{cauchy_problem_gelfand-shilov}
can be obtained imposing some control on the behaviour of the coefficients for $|x|$ large. Concerning $H^\infty$-well-posedness, sufficient conditions for well-posedness have been proved in \cite {ABZsuff} whereas necessary conditions have been stated in \cite{ABZnec}. Analogous results have been obtained in Gevrey spaces  by the first author et al. in the recent papers \cite{AACpevolGevreynec} and \cite{AACM25}. 

Starting from the latter results, the aim of this paper is to investigate the Cauchy problem \eqref{cauchy_problem_gelfand-shilov} for the operator \eqref{differential_p_evolution_operator} in the case when the initial data belong to the Gelfand-Shilov space $\mathcal{S}_s^\theta(\R)$ for some $s>1, \theta >1.$ This space is
defined as the space of all functions $f \in C^\infty(\R)$ such that
\begin{equation}\label{GSdef}\sup_{\alpha \in \N_0}\sup_{x \in \R} C^{-\alpha}(\alpha!)^{-\theta}\exp(c|x|^{1/s}) |f^{(\alpha)}(x)|<\infty
	\end{equation}
for some positive constants $C,c$. As it is clear from \eqref{GSdef}, the elements of $\mathcal{S}_s^\theta(\R)$ are Gevrey regular functions which decay exponentially at infinity. Due to their relation with Gevrey classes and to their good properties concerning the action of Fourier transform $\mathscr{F}$, these spaces have been largely employed in the last twenty years to treat partial differential equations via Fourier and microlocal analysis, see \cite{Arias_GS, AAC3evolGelfand-Shilov, AW24, ACJMPA, scncpp2, Cappiello, CappielloRodino, CGR1, CGR2, CGR3}. Concerning in particular $p$-evolution equations, we mention the recent papers by the first author et al. \cite{AAC3evolGelfand-Shilov, ACJMPA} where existence results have been obtained in Gelfand-Shilov spaces for the cases $p=2$ and $p=3$. Recently, in \cite{Arias_GS} the author obtained a well-posedness result for Schr\"odinger-type operators of the form
$$P(t,x,D_t,D_x) = D_t - \Delta_x + \sum_{j=1}^n a_j(t,x)D_{x_j} + a_0(t,x), \qquad (t,x) \in [0,T]\times \R^n.$$ Assuming $a_j$ Gevrey regular of order $\theta_0$ for some $\theta_0 >1$ and $|\Im a_j(t,x)| \leq C_j \japx^{-\sigma}$ for some $\sigma \in (0, 1)$ he proved that the related Cauchy problem is well-posed  in $S^\theta_s(\R^n)$,  for $\theta \in [\theta_0, \min\{1/(1-\sigma),s\}]$ and also proved the sharpness of this condition. The aim of the present paper is to obtain similar results for operators of the form \eqref{differential_p_evolution_operator}.  
\\
\indent In order to state our main results, we need to introduce a scale of weighted Sobolev-type spaces related to Gelfand-Shilov spaces. Fixed $\rho=(\rho_1,\rho_2),m=(m_1,m_2) \in \R^2, s >1, \theta >1$ we define 
\begin{equation}\label{wGSobolev}
H_{\rho;s,\theta}^{m}(\R) = \{ u \in L^2(\R) \ : \ \japx^{m_2} \langle D \rangle^{m_1} e^{\rho_2 \japx^{1/s}} e^{\rho_1 \langle D \rangle^{1/\theta}}u \in L^2(\R) \}.
\end{equation}
It is easy to verify that 
$$\mathcal{S}_s^\theta(\R)= \bigcup_{\stackrel{\rho \in \R^2}{ \rho_1>0, \rho_2 >0}}H_{\rho;s,\theta}^{m}(\R)$$
for every $m \in \R^2$.

\begin{definition}\label{definition_gelfand-shilov-sobolev_well-posedness}\textup{
We say that the Cauchy problem \eqref{cauchy_problem_gelfand-shilov} is \textit{well-posed} in $\mathcal{S}^\theta_s(\R)$ if for any given $m=(m_1,m_2),\rho=(\rho_1,\rho_2) \in \R^2$, with $\rho_j>0,j=1,2$, there exist $\tilde{\rho}=(\tilde{\rho}_1,\tilde{\rho}_2) \in \R^2$ with $\tilde \rho_j>0, j=1,2,$ and a constant $C:=C(\rho,T)>0$ such that, for all $f \in C \left( [0,T];H_{\rho;s,\theta}^{m}(\R) \right)$ and $g \in H_{\rho;s,\theta}^{m}(\R)$, there exists a unique solution $u \in C^1 \left( [0,T];H_{\tilde{\rho};s,\theta}^{m}(\R) \right)$ of \eqref{cauchy_problem_gelfand-shilov} and the following energy estimate holds
$$
\| u(t,\cdot) \|_{H_{\tilde{\rho};s,\theta}^{m}} \leq C \left( \| g \|_{H_{\rho;s,\theta}^{m}}^2 + \int_0^t \| f(\tau,\cdot) \|_{H_{\rho;s,\theta}^{m}}^2 d\tau \right).
$$
}\end{definition}

It is easy to verify, eventually conjugating our operator by $\japx^{m_{2}} \langle D \rangle^{m_{1}}$, that we can replace $m \in \R^2$ by $(0,0)$ in Definition \ref{definition_gelfand-shilov-sobolev_well-posedness}; hence, from now on, we shall limit to consider the spaces \eqref{wGSobolev} for $m=(0,0)$, and denote them simply by $H^0_{\rho;s,\theta}(\R)$. The main result that will be proved in this paper, is the following

\begin{theorem}\label{theorem_main_result_3}
Let $\theta_0>1$ and $\sigma \in \left( \frac{p-2}{p-1},1 \right)$ such that $\theta_0 < \frac{1}{(p-1)(1-\sigma)}$. Let $P$ be an operator of the type \eqref{differential_p_evolution_operator} whose coefficients satisfy the following assumptions:
\begin{itemize}
\item[\textup{(i)}] $a_p \in C \left( [0,T];\R \right)$ and there exists $C_{a_p}>0$ such that $|a_p(t)| \geq C_{a_p}$, for all $t \in [0,T]$.
\item[\textup{(ii)}] $| \partialx^\beta a_{p-j}(t,x) | \leq C_{a_{p-j}}^{\beta + 1} \beta!^{\theta_0} \japx^{-\frac{p-j}{p-1}\sigma - \beta}$, for some $C_{a_{p-j}}>0$, $j=1,...,p-1$ and for all $\beta \in \mathbb{N}_0$, $(t,x) \in [0,T] \times \R$.
\end{itemize}
Let $s,\theta>1$ such that $(p-1)\theta < \min \left\lbrace \frac{1}{1-\sigma}, s \right\rbrace$ and $\theta \geq \theta_0$, and let $f \in C \left( [0,T];H_{\rho;s,\theta}^0(\R) \right)$ and $g \in H_{\rho;s,\theta}^0(\R)$, for some $\rho=(\rho_1,\rho_2) \in \R^2$ with $\rho_1,\rho_2 > 0$. Then there exists a unique solution $u \in C^1 \left( [0,T];H_{(\tilde{\rho}_1,\delta);s,\theta}^0(\R) \right)$ of \eqref{cauchy_problem_gelfand-shilov} for some $\tilde{\rho}_1 \in (0,\rho_1)$, $\delta \in (0,\rho_2)$ and it satisfies the energy estimate
\begin{equation}\label{energy_estimate_cauchy_gelfand-shilov}
\| u(t) \|_{H_{(\tilde{\rho}_1,\delta);s,\theta}^0}^2 \leq C \left( \| g \|_{H_{\rho;s,\theta}^0}^2 + \int_0^t \| f(\tau) \|_{H_{\rho;s,\theta}^0}^2 d\tau \right),
\end{equation}
for all $t \in [0,T]$ and for some constant $C>0$. In particular, the Cauchy problem \eqref{cauchy_problem_gelfand-shilov} is well-posed in $\mathcal{S}_s^\theta(\R)$.
\end{theorem}
The proof of the above result is based on the reduction of the Cauchy problem \eqref{cauchy_problem_gelfand-shilov} to an equivalent Cauchy problem with data in Gevrey spaces for which we proved well-posedness in \cite{AACM25}.  
\\ \indent
The second part of the paper is devoted to prove that the case $(p-1)\theta = \min \left\lbrace \frac{1}{1-\sigma}, s \right\rbrace$ is a critical threshold for the well-posedness in $\mathcal{S}^\theta_s(\R)$. Namely we shall test on suitable models that if $(p-1)\theta > \min \left\lbrace \frac{1}{1-\sigma}, s \right\rbrace$, the Cauchy problem \eqref{cauchy_problem_gelfand-shilov} is not well-posed in $\mathcal{S}^\theta_s(\R)$ in general. The critical case remains an open problem, cf. Remark \ref{criticalcase} below.
\\ \indent
The paper is organized as follows. In Section \ref{Background} we recall a result concerning well-posedness in Gevrey spaces for the problem \eqref{cauchy_problem_gelfand-shilov} and explain the idea of the proof of Theorem \ref{theorem_main_result_3} which is based on a suitable conjugation of the operator $P$. In Section \ref{section_conjugation} we develop this conjugation and give an estimate of the coefficients of the conjugated operator. In Section \ref{section_proofthm1} we prove Theorem \ref{theorem_main_result_3}. Finally, in Section \ref{section_ill-posedness} we exhibit examples of operators satisfying the assumptions of Theorem \ref{theorem_main_result_3} for which the Cauchy problem \eqref{cauchy_problem_gelfand-shilov} is ill-posed if $(p-1)\theta > \min \left\lbrace \frac{1}{1-\sigma}, s \right\rbrace$.
\\

\section{Background theory and idea of the proof}\label{Background}
As we said in the Introduction, the well-posedness of \eqref{cauchy_problem_gelfand-shilov} in Gelfand-Shilov spaces  can be proved using a similar result proved in \cite{AACM25}
for the same initial value problem with data $f,g$ in Gevrey-type spaces. In order to recall this result, let us define, for fixed $m \in \R, \rho>0, \theta >1$ the space
$$H^m_{\rho;\theta}(\R) =\{u \in L^2(\R): \langle D \rangle^m e^{\rho \langle D \rangle^{1/\theta}}u \in L^2(\R)\}.$$
Notice that this space can be viewed as a particular case of the weighted Gevrey-Sobolev space defined by \eqref{wGSobolev} obtained for $ m_2 =\rho_2=0, m_1=m$ and $\rho_1=\rho.$ 
Moreover, we set $\mathcal{H}^\infty_\theta (\R):= \bigcup_{\rho >0}H^m_{\rho;\theta}(\R)$. Then we have the following result, cf. \cite[Theorem 1.1]{AACM25}.
 
\begin{theorem}\label{gevreythm}
	Let $P(t,x,D_t,D_x)$ be an operator of the form \eqref{differential_p_evolution_operator} whose coefficients satisfy the assumptions (i) and (ii) of Theorem \ref{theorem_main_result_3} for some $\theta_0>1$ and $\sigma \in \left(\frac{p-2}{p-1},1\right)$ such that $\theta_0 < \frac{1}{(p-1)(1-\sigma)}$. 
	Let $f \in C\left([0,T];H_{\rho;\theta}^m(\R)\right)$ and $g \in H_{\rho;\theta}^m(\R)$ for some $m,\rho, \theta \in\R$ with $\rho>0$, and $\theta \in \left[ \theta_0, \frac{1}{(p-1)(1-\sigma)} \right)$. Then the Cauchy problem \eqref{cauchy_problem_gelfand-shilov} admits a unique solution $u \in C\left([0,T];H_{\tilde \rho;\theta}^m(\R)\right)$ for some $\tilde \rho \in (0,\rho)$, and the solution satisfies the energy estimate
	\begin{equation}\label{energy_estimate_cauchy_gevrey}
		\| u(t) \|_{H_{\tilde \rho;\theta}^m}^2 \leq C \left( \| g \|^2_{H_{\rho;\theta}^m} + \int_0^t \| f(\tau) \|_{H_{\rho;\theta}^m}^2 d\tau \right),
	\end{equation}
	for all $t\in[0,T]$ and for some constant $C>0$. In particular, for $\theta\in\left[\theta_0,\frac{1}{(p-1)(1-\sigma)}\right)$ the Cauchy problem \eqref{cauchy_problem_gelfand-shilov} is well-posed in $\mathcal{H}_\theta^\infty(\R)$.
\end{theorem}

The idea of the proof of Theorem \ref{theorem_main_result_3} is to reduce \eqref{cauchy_problem_gelfand-shilov} to an initial value problem with data in Gevrey spaces using the change of variable $$v(t,x)= e^{\delta \langle x \rangle^{1/s}}u(t,x)$$ for some $\delta \in (0,\rho_2)$. This change of variable pulls back the data space for $f$ and $g$ to some Gevrey-Sobolev space and leads to a new Cauchy problem for the conjugated operator $P_\delta:=e^{\delta\japx^{1/s}} P e^{-\delta\japx^{1/s}}$ with data 
$$
e^{\delta\japx^{1/s}} f(t,x) \in C([0,T], H_{\rho_1;\theta}^0(\R)), \ e^{\delta\japx^{1/s}} g(x) \in H_{\rho_1;\theta}^0(\R).
$$
Namely, we are reduced to consider the problem
\begin{equation}\label{auxiliary_cauchy_problem_gelfand-shilov}
	\left\lbrace\begin{array}{l}
		P_\delta(t,x,D_t,D_x) v(t,x) = e^{\delta\japx^{1/s}} f(t,x) \\ 
		v(0,x) = e^{\delta\japx^{1/s}} g(x)
	\end{array} \right., \quad (t,x) \in [0,T] \times \R. 
\end{equation}
In order to apply Theorem \ref{gevreythm} to this problem, we need to verify that the new operator $P_\delta$ is of the same form of $P$ and satisfies the assumptions of Theorem \ref{gevreythm}. Then we obtain that \eqref{auxiliary_cauchy_problem_gelfand-shilov} admits a unique solution $v \in C\left([0,T];H_{\tilde \rho_1;\theta}^0(\R)\right)$ for some $\tilde \rho_1 \in (0,\rho_1)$. Once we have this, it is easy to show that $u(t,x)=e^{-\delta \japx^{1/s}}v(t,x) \in C\left([0,T];H_{(\tilde \rho_1, \delta);s,\theta}^m(\R)\right)$ is the unique solution of \eqref{cauchy_problem_gelfand-shilov} and satisfies
\eqref{energy_estimate_cauchy_gelfand-shilov} for $\tilde \rho_2 = \delta$.
The next section is devoted to study the conjugated operator $P_\delta$ and to give a suitable estimate of its coefficients.

\section{Conjugation of $P$ by $e^{\delta\japx^{1/s}}$}\label{section_conjugation}

In this section we perform the conjugation of the operator $P(t,x,D_t,D_x)$ by $e^{\delta\japx^{1/s}}$ and its inverse, where $\delta \in (0,\rho_2)$. The conjugation will be made term-by-term, and the next steps are dedicated to this purpose.

\begin{itemize}

\item \textbf{The conjugation of $D_t$.} Since $e^{\delta\japx^{1/s}}$ does not depend on $t$, the conjugation is trivially given by
$$
e^{\delta\japx^{1/s}} \circ D_t \circ e^{-\delta\japx^{1/s}} = D_t.
$$

\item \textbf{The conjugation of $a_p(t)D_x^p$.} Using Leibniz formula, we get
\begin{eqnarray}
e^{\delta\japx^{1/s}} D_x^p e^{-\delta\japx^{1/s}} & = & D_x^p + \mathbf{op} \left( \sum_{k=1}^p \frac{1}{k!} \cdot \frac{p!}{(p-k)!} \xi^{p-k} e^{\delta\japx^{1/s}} D_x^k e^{-\delta\japx^{1/s}} \right) \nonumber \\
& = & D_x^p + \sum_{k=1}^p b_{p-k}^{(\delta)}(x) D_x^{p-k}, \nonumber
\end{eqnarray}
where
$$
b_{p-k}^{(\delta)}(x) := \binom{p}{k} e^{\delta\japx^{1/s}} D_x^k e^{-\delta\japx^{1/s}}.
$$
Now, let us estimate the terms $b_{p-k}^{(\delta)}$. By Fa{\`a} di Bruno formula, it follows that
$$
e^{\delta\japx^{1/s}} D_x^k e^{-\delta\japx^{1/s}} = \sum_{\ell = 1}^k \frac{1}{\ell !}\sum_{\stackrel{k_1 + \cdots + k_\ell = k}{k_\nu \geq 1}} \frac{k!}{k_1! \cdots k_\ell !} \prod_{\nu = 1}^\ell D_x^{k_\nu} \left(-\delta \japx^{1/s} \right).
$$
Notice that, since $|\partialx^\beta \japx^m| \leq C_m^\beta \beta! \japx^{m-\beta}$ for all $m \in \R$, we obtain the estimate
\begin{eqnarray}
\left| \prod_{\nu = 1}^\ell D_x^{k_\nu} \left( -\delta\japx^{1/s} \right) \right| & = & \delta^\ell \prod_{\nu = 1}^\ell \left| D_x^{k_\nu} \japx^{1/s} \right| \leq \delta^\ell \prod_{\nu = 1}^\ell C^{k_\nu} k_\nu ! \japx^{\frac{1}{s}-k_\nu} \nonumber \\
& \leq & C_\delta^{k+1} k_1! \cdots k_\ell ! \japx^{k \left( \frac{1}{s} - 1 \right)}. \nonumber
\end{eqnarray}
Hence,
$$
| b_{p-k}^{(\delta)}(x) | \leq \frac{p!}{k!(p-k)!} \sum_{\ell = 1}^k \frac{1}{\ell !} \sum_{\stackrel{k_1 + \cdots + k_\ell =k}{k_\nu \geq 1}} k! C_\delta^{k+1} \japx^{k \left( \frac{1}{s} - 1 \right)},
$$
which gives us that $b_{p-k}^{(\delta)} \sim \japx^{k \left( \frac{1}{s} - 1 \right)}$. Similarly we obtain that
\begin{equation}\label{estimate_b_p-k}
| D_x^\beta b_{p-k}^{(\delta)}(x) | \leq \tilde{C}_{\delta}^{\beta + k + 1} \beta! \japx^{k \left( \frac{1}{s} - 1 \right) - \beta},
\end{equation}
for all $\beta \in \mathbb{N}_0$. Therefore, the conjugation of $a_p(t)D_x^p$ is given by
$$
e^{\delta\japx^{1/s}} (a_p(t) D_x^p) e^{-\delta\japx^{1/s}} = a_p(t) D_x^p + \sum_{k=1}^p \tilde{a}_{p-k}^{(\delta)}(t,x) D_x^{p-k},
$$
where $\tilde{a}_{p-k}^{(\delta)}(t,x) = a_p(t) b_{p-k}^{(\delta)}(x) \in C([0,T]; \mathcal{B}^\infty(\R))$ and satisfy, for every $k=1,...,p$, the following estimate:
\begin{equation}\label{estimate_tilde_a_p-k}
\sup_{t \in [0,T]} | D_x^\beta \tilde{a}_{p-k}^{(\delta)}(t,x) | \leq C_{\delta,1}^{\beta + 1} \beta! \japx^{k \left( \frac{1}{s} -1 \right) - \beta}.
\end{equation}

\item \textbf{The conjugation of $a_{p-j}(t,x)D_x^{p-j}$, $j=1,...,p-1$.} For each $j=1,...,p-1$, by using again Leibniz formula we get
\begin{eqnarray}
e^{\delta\japx^{1/s}} (a_{p-j}(t,x)D_x^{p-j}) e^{-\delta\japx^{1/s}} & = & a_{p-j}(t,x)D_x^{p-j} \nonumber \\
& + & \sum_{\ell=1}^{p-j} \frac{1}{\ell!} a_{p-j}(t,x) \frac{(p-j)!}{(p-j-\ell)!} e^{\delta\japx^{1/s}} D_x^\ell e^{-\delta\japx^{1/s}} D_{x}^{p-j-\ell}\nonumber \\
& = & a_{p-j}(t,x)D_x^{p-j} + \sum_{\ell=1}^{p-j} a_{p-j}(t,x) c_{p-j-\ell}^{(\delta)}(x) D_x^{p-j-\ell}, \nonumber
\end{eqnarray}
where
$$
c_{p-j-\ell}^{(\delta)}(x) = \binom{p-j}{\ell} e^{\delta\japx^{1/s}} D_x^\ell e^{-\delta\japx^{1/s}}.
$$
By the same argument used before, it follows that $c_{p-j-\ell}^{(\delta)}$ satisfy \eqref{estimate_b_p-k}, for every $j=1,...,p-1$ and $\ell=1,...,p-j$.

\item \textbf{The conjugation of $a_0(t,x)$.} It is simply given by
$$
e^{\delta\japx^{1/s}} a_0(t,x) e^{-\delta\japx^{1/s}} = a_0(t,x).
$$

\end{itemize}

By the previous analysis, we can assert that the operator $P_\delta(t,x,D_t,D_x)$ can be written as
\begin{eqnarray}
 P_\delta (t,x,D_t,D_x) & = & D_t + a_p(t)D_x^p + \sum_{j=1}^{p-1} a_{p-j}(t,x)D_x^{p-j} + a_0(t,x) \nonumber \\
& + & \sum_{k=1}^{p-1} \tilde{a}_{p-k}^{(\delta)}(t,x)D_x^{p-k} + \tilde{a}_0^{(\delta)}(t,x) + \sum_{j=1}^{p-1} \sum_{\ell=1}^{p-j} a_{p-j}(t,x) c_{p-j-\ell}^{(\delta)}(x) D_x^{p-j-\ell} \nonumber.
\end{eqnarray}
Notice that the double sum in the last term can be expressed as 
$$
\sum_{j=1}^{p-1} a_{p-j}(t,x) \sum_{\ell=1}^{p-j} c_{p-j-\ell}^{(\delta)}(x) D_x^{p-j-\ell} =  \sum_{k=2}^{p} d_{p-k}^{(\delta)}(t,x) D_x^{p-k},
$$
where
$$
d_{p-k}^{(\delta)} (t,x):= c_{p-k}^{(\delta)}(x) \sum_{\ell = k-1}^p a_{p-\ell}(t,x), \quad k=2,...,p.
$$
In particular, we have that $d_{p-k}^{(\delta)}, k=2,\ldots, p,$ satisfy \eqref{estimate_tilde_a_p-k}. Hence,
\begin{eqnarray}
e^{\delta\japx^{1/s}} \circ P \circ e^{-\delta\japx^{1/s}} & = & D_t + a_p(t) D_x^p + a_{p-1}(t,x) D_x^{p-1} + \sum_{j=2}^{p-1} a_{p-j}	(t,x) D_x^{p-j} + a_0(t,x) \nonumber \\
& + & \tilde{a}_{p-1}^{(\delta)}(t,x) D_x^{p-1} + \sum_{k=2}^{p-1} \tilde{a}_{p-k}^{(\delta)}(t,x) D_x^{p-k} + \tilde{a}_0^{(\delta)}(t,x) \nonumber \\
& + & \sum_{k=2}^{p-1} d_{p-k}^{(\delta)} (t,x) D_x^{p-k} + d_0^{(\delta)}(t,x), \nonumber
\end{eqnarray}
and, finally,
\begin{eqnarray}\label{first_conjugation_structure_gelfand-shilov}
P_\delta(t,x,D_t,D_x) & = & D_t + a_p(t) D_x^p + \left\lbrace a_{p-1}(t,x) + \tilde{a}_{p-1}^{(\delta)}(t,x) \right\rbrace D_x^{p-1} \nonumber \\
& + & \sum_{j=2}^{p} \left\lbrace a_{p-j}(t,x) + \tilde{a}_{p-j}^{(\delta)}(t,x) + d_{p-j}^{(\delta)}(t,x) \right\rbrace D_x^{p-j} \nonumber \\
& = & D_t + ia_p(t) D_x^p + \sum_{j=1}^p {a}_{p-j}^{(\delta)}(t,x) D_x^{p-j},
\end{eqnarray}
where
$$
\begin{array}{l}
a_{p-1}^{(\delta)}(t,x) := a_{p-1}(t,x) + \tilde{a}_{p-1}^{(\delta)}(t,x), \\ 
a_{p-j}^{(\delta)}(t,x) := a_{p-j}(t,x) + \tilde{a}_{p-j}^{(\delta)}(t,x) + d_{p-j}^{(\delta)}(t,x), \quad 2 \leq j \leq p.
\end{array} 
$$

From the previous estimates we obtain that $a_{p-j}^{(\delta)}$ satisfy the following estimate:
\begin{equation}\label{estimate_a_p-j^delta}
\sup_{t \in [0,T]} \left( | D_x^\beta \tilde{a}_{p-j}^{(\delta)}(t,x) | + | D_x^\beta d_{p-j}^{(\delta)}(t,x) | \right) \leq C_{\delta}^{\beta + 1} \beta!^{\theta_0} \japx^{-j \left( 1 - \frac{1}{s} \right) - \beta}.
\end{equation}

We are now ready to prove Theorem \ref{theorem_main_result_3}.

\section{Proof of Theorem \ref{theorem_main_result_3}} \label{section_proofthm1}

To prove Theorem \ref{theorem_main_result_3}, we shall apply Theorem \ref{gevreythm} to the operator $P_\delta$, so we need to show that the operator $P_\delta$ satisfies the assumptions of the latter theorem. With this purpose we distinguish two cases.

\textbf{The case $s \geq \frac{1}{1-\sigma}$.} In this case, by the assumption of Theorem \ref{theorem_main_result_3} we have
$$
\theta < \frac{1}{(p-1)(1-\sigma)} \leq \frac{s}{p-1}.
$$   
In this case we observe that the coefficients of $P_\delta$ satisfy the assumptions of Theorem \ref{gevreythm}, that is,
$$
j \left( 1 - \frac{1}{s} \right) \geq \frac{p-j}{p-1}\sigma, \quad \forall j = 1,..., p.
$$
In fact, we notice that as $j$ increases, the left-hand side of the above inequality increases and the right-hand side decreases, so this inequality holds for all $j=1,...,p$ if, and only if, it holds for $j=1$, that is, 
$$
\left( 1 - \frac{1}{s} \right) \geq \frac{p-1}{p-1}\sigma \quad \Leftrightarrow \quad 1 - \frac{1}{s} \geq \sigma \quad \Leftrightarrow \quad s \geq \frac{1}{1-\sigma},
$$
which is true. Hence, we have
\begin{equation}\label{estimate_sup_a_p-j^delta}
\sup_{t \in [0,T]} | D_x^\beta a_{p-j}^{(\delta)} (t,x) | \leq C_{\delta}^{\beta+1} \beta!^{\theta_0} \japx^{-\frac{p-j}{p-1}\sigma-\beta}.
\end{equation}
Let us then consider the Cauchy problem \eqref{cauchy_problem_gelfand-shilov} with data
$$
f \in C \left( [0,T];H_{\rho;s,\theta}^0(\R) \right) \quad \text{and} \quad g \in H_{\rho;s,\theta}^0(\R),
$$
for $\rho=(\rho_1,\rho_2)$, with $\rho_1,\rho_2>0$. If $\delta \in (0,\rho_2)$, then
$$
f_\delta := e^{\delta\japx^{1/s}}f \in C \left( [0,T];H_{(\rho_1,\rho_2-\delta);s,\theta}^0(\R) \right) \quad \text{and} \quad g_\delta := e^{\delta\japx^{1/s}}g \in H_{(\rho_1,\rho_2-\delta);s,\theta}^0(\R),
$$
because, if $\phi \in H_{\rho;s,\theta}^0(\R)$, then 
\begin{equation}\label{phi_in_gelfand-shilov-sobolev}
\phi_{(\rho_1,\rho_2)}(x) := e^{\rho_2\japx^{1/s}} e^{\rho_1\langle D \rangle^{1/\theta}} \phi \in L^2(\R),
\end{equation}
and
\begin{eqnarray}
e^{(\rho_2-\delta)\japx^{1/s}} e^{\rho_1 \langle D \rangle^{1/\theta}} \underbrace{e^{\delta \japx^{1/s}} \phi}_{:= \phi_\delta} & = & e^{(\rho_2 - \delta) \japx^{1/s}} e^{\rho_1 \langle D \rangle^{1/\theta}} e^{\delta \japx^{1/s}} \left( e^{\rho_2 \japx^{1/s}} e^{\rho_1 \langle D \rangle^{1/\theta}} \right)^{-1} \underbrace{e^{\rho_2 \japx^{1/s}} e^{\rho_1 \langle D \rangle^{1/\theta}} \phi}_{=\phi_{(\rho_1,\rho_2)}} \nonumber \\
& = & \underbrace{e^{(\rho_2 - \delta) \japx^{1/s}} e^{\rho_1 \langle D \rangle^{1/\theta}} e^{\delta \japx^{1/s}} e^{-\rho_1 \langle D \rangle^{1/\theta}} e^{-\rho_2 \japx^{1/s}}}_{:=A_{(\rho_1,\rho_2,\delta)}} \phi_{(\rho_1,\rho_2)} \in L^2(\R), \nonumber
\end{eqnarray}
since $A_{(\rho_1,\rho_2,\delta)}$ has order zero and $\phi_{(\rho_1,\rho_2)} \in L^2(\R)$, and this means that $\phi_\delta \in H_{(\rho_1,\rho_2-\delta);s,\theta}^0(\R)$.

Since $\rho_2-\delta>0$, it follows that 
$$
f_\delta \in C \left( [0,T];H_{\rho_1;\theta}^0(\R) \right) \quad \text{and} \quad g_\delta \in H_{\rho_1;\theta}^0(\R).
$$
Since we are considering $\theta \in \left[ \theta_0, \min \left\lbrace \frac{1}{(p-1)(1-\sigma)}, \frac{s}{p-1} \right\rbrace \right)$, we obtain that the auxiliary Cauchy problem given by \eqref{auxiliary_cauchy_problem_gelfand-shilov} is well-posed in $\mathcal{H}_\theta^\infty(\R)$, that is, there exists a unique solution $v \in C^1 \left( [0,T];H_{\tilde{\rho}_1;\theta}^0(\R) \right)$ of \eqref{auxiliary_cauchy_problem_gelfand-shilov}, with initial data $f_\delta$ and $g_\delta$, satisfying
\begin{equation}\label{energy_estimate_auxiliary_gelfand-shilov}
\| v(t,\cdot) \|_{H_{\tilde{\rho}_1;\theta}^0} \leq C(\delta,\rho_1,T) \left( \| g_\delta \|_{H_{\rho_1;\theta}^0}^2 + \int_0^t \| f_\delta(\tau,\cdot) \|_{H_{\rho_1;\theta}^0}^2 d\tau \right),
\end{equation}
for some $\tilde{\rho}_1 \in (0,\rho_1)$. Now we set $u(t,x) := e^{-\delta \japx^{1/s}} v(t,x)$. Notice that $u \in C^1 \left( [0,T]; H_{(\tilde{\rho}_1,\delta);s,\theta}(\R) \right)$ and it solves the Cauchy problem \eqref{cauchy_problem_gelfand-shilov}, because
$$
u(0,x) = e^{-\delta \japx^{1/s}} v(0,x) = e^{-\delta \japx^{1/s}} g_\delta(x) = g(x),
$$
and $P_\delta(t,x,D_t,D_x) v(t,x) = f_\delta(t,x)$ implies that
$$
e^{\delta \japx^{1/s}} P(t,x,D_t,D_x) \underbrace{e^{-\delta \japx^{1/s}} v(t,x)}_{= u(t,x)} = e^{\delta \japx^{1/s}} f(t,x) \quad \Rightarrow \quad P(t,x,D_t,D_x) u(t,x) = f(t,x).
$$
Finally, we can use estimate \eqref{energy_estimate_auxiliary_gelfand-shilov} to obtain
\begin{eqnarray}
\| u(t,\cdot) \|_{H_{(\tilde{\rho}_1,\delta);s,\theta}^0} & \leq & C(\delta) \| v(t,\cdot) \|_{H_{\tilde{\rho}_1;\theta}^0} \nonumber \\
& \leq & C(\rho_1,\delta,T) \left( \| g_\delta \|_{H_{\rho_1;\theta}^0}^2 + \int_0^t \| f_\delta(\tau,\cdot) \|_{H_{\rho_1;\theta}^0}^2 d\tau \right) \nonumber \\
& \leq & C(\rho_1,\delta,T) \left( \| g \|_{H_{(\rho_1,\delta);s,\theta}^0}^2 + \int_0^t \| f(\tau,\cdot) \|_{H_{(\rho_1,\delta);s,\theta}^0}^2 d\tau \right) \nonumber \\
& \leq & C(\rho_1,\delta,T) \left( \| g \|_{H_{\rho;s,\theta}^0}^2 + \int_0^t \| f(\tau,\cdot) \|_{H_{\rho;s,\theta}^0}^2 d\tau \right), \nonumber
\end{eqnarray}
where $C(\rho_1,\delta,T)$ is a positive constant depending on $\rho_1$, $\delta$ and $T$.

To prove the uniqueness of the solution, let us consider two solutions
$$
u_j \in C^1 \left( [0,T];H_{(\tilde{\rho}_1,\delta);s,\theta}^0(\R) \right), \quad j=1,2,
$$
with $\delta<\rho_2$, for the Cauchy problem \eqref{cauchy_problem_gelfand-shilov}. By taking any $\tilde{\delta} \in (0,\delta)$, notice that
$$
v_j := e^{\tilde{\delta} \japx^{1/s}} u_j, \quad j=1,2,
$$
are solutions of \eqref{auxiliary_cauchy_problem_gelfand-shilov}, for $\delta$ replaced by $\tilde{\delta}$, and also $v_j \in C^1 \left( [0,T];H_{\tilde{\rho}_1;\theta}^0(\R) \right)$. The well-posedness in $\mathcal{H}_\theta^\infty(\R)$ of \eqref{auxiliary_cauchy_problem_gelfand-shilov} gives us $v_1=v_2$, hence, $e^{\tilde{\delta} \japx^{1/s}} u_1 = e^{\tilde{\delta} \japx^{1/s}} u_2$, which implies that $u_1=u_2$.
\\

\textbf{The case $s < \frac{1}{1-\sigma}$.} The assumption of Theorem \ref{theorem_main_result_3} becomes
$$
\theta_0 \leq \theta < \frac{s}{p-1} < \frac{1}{(p-1)(1-\sigma)}.
$$
In this case we observe that $\japx^{-\frac{p-j}{p-1}\sigma} \leq \japx^{-\frac{p-j}{p-1}\left(1-\frac{1}{s}\right)}$ and $\japx^{-j\left(1-\frac{1}{s}\right)} \leq \japx^{-\frac{p-j}{p-1}\left(1-\frac{1}{s}\right)}$, so we get
$$
\sup_{t \in [0,T]} | D_x^\beta a_{p-j}^{(\delta)}(t,x) | \leq C^{\beta+1} \beta!^{\theta_0} \japx^{-\frac{p-j}{p-1}(1-\frac{1}{s})-\beta}.
$$
Then we can repeat the same argument used in the first case with $\sigma$ replaced by $1-\frac{1}{s}$. Notice that $1-\frac{1}{s} \in \left( \frac{p-2}{p-1},1 \right)$ since
$$
1-\frac{1}{s} > \frac{p-2}{p-1} \quad \Leftrightarrow \quad \frac{1}{s} < \frac{1}{p-1} \quad \Leftrightarrow \quad s > p - 1
$$
which holds true since $\frac{s}{p-1}>\theta>1$. Then in this case the Cauchy problem is well-posed in $\mathcal{S}^\theta_s(\R)$ for $\theta_0 \leq \theta < \frac{1}{(p-1)\left(1-1+\frac{1}{s}\right)}=\frac{s}{p-1}=\min \left\lbrace \frac{1}{(p-1)(1-\sigma)},\frac{s}{p-1} \right\rbrace$. This concludes the proof of Theorem \ref{theorem_main_result_3}.
\begin{flushright}
$\qed$
\end{flushright}

%


\section{Ill-posedness results for model operators}\label{section_ill-posedness}

In the previous section we have proved Theorem \ref{theorem_main_result_3} under the assumption
$$
(p-1)\theta < \min \left\lbrace \frac{1}{1-\sigma}, s \right\rbrace.
$$
The aim of this section is to prove, by counterexamples, that if
\begin{equation}\label{inequality_sharpness}
(p-1)\theta > \min \left\lbrace \frac{1}{1-\sigma}, s \right\rbrace,
\end{equation}
then the Cauchy problem \eqref{cauchy_problem_gelfand-shilov} is not well-posed in $\mathcal{S}_{s}^{\theta}(\R)$, in general, for an operator $P$ satisfying the conditions (i) and (ii) of Theorem \ref{theorem_main_result_3}. With this purpose we distinguish two cases.
\\

\textbf{The case $s \leq \frac{1}{1-\sigma}$.} In this case, \eqref{inequality_sharpness} turns into the condition $s < (p-1)\theta$. Then we have the following result.

\begin{proposition}\label{cauchy_problem_similar_schrodinger}
Let $s,\theta>1$ such that $s < (p-1)\theta$. Then there exists $g \in \mathcal{S}^\theta_s(\R)$ such that the Cauchy problem
\begin{equation} \label{modelSchrodinger}
\left\lbrace \begin{array}{l}
(D_t + D_x^p) u = 0 \\ 
u(0,x) = g(x)
\end{array}, \quad (t,x) \in [0,T] \times \R,  \right.
\end{equation}
admits a unique solution which is not in $\mathcal{S}^\theta_s(\R)$.
\end{proposition}
\begin{proof}
The proof relies on the application of Theorem 1.2 in \cite{AW24} which states that for every polynomial $q(\xi)$ of degree $p$ with real coefficients and for every $s,\theta>1$ such that $1<s<(p-1)\theta$ there exists $\varphi \in \mathcal{S}_\theta^s(\R)$ such that
\begin{equation}\label{ale_patrik}
e^{iq(\xi)} \varphi(\xi) \notin \mathcal{S}_\theta^s(\R).
\end{equation}
Taking $q(t,\xi):=-t\xi^p$ for fixed $t>0$ and $\varphi \in \mathcal{S}^s_\theta(\R)$ such that \eqref{ale_patrik} holds and choosing $g=\mathscr{F}^{-1}(\varphi)\in \mathcal{S}_s^\theta(\R)$ as initial datum, by elementary arguments we have that the solution of \eqref{modelSchrodinger} is
$$
u_g(t,x) = \mathscr{F}^{-1} (e^{-it \xi^p } \widehat{g}(\xi)) = \mathscr{F}^{-1} (e^{-it \xi^p } \varphi(\xi)) \notin \mathcal{S}^\theta_s(\R).
$$
\end{proof}

The operator $D_t+D_x^p$ obviously satisfies the assumptions (i), (ii) for every $\sigma \in \left( \frac{p-2}{p-1},1 \right)$. Hence, assuming $\sigma$ close to $1$, we are in the situation $s\leq\frac{1}{1-\sigma}$. Then we have proved that if $s<(p-1)\theta$, the Cauchy problem \eqref{cauchy_problem_gelfand-shilov} is not well-posed in general.
\\

\textbf{The case $s > \frac{1}{1-\sigma}$.} Now \eqref{inequality_sharpness} turns into $\frac{1}{1-\sigma} < (p-1)\theta$. Then we can consider the operator
\begin{equation}\label{model_operator}
M = D_t + D_x^p + i \japx^{-\sigma} D_x^{p-1}
\end{equation}
and prove the following result.

\begin{proposition}\label{proposition_main_result_4}
Let $M$ be the operator in \eqref{model_operator}, with $\sigma \in \left( \frac{p-2}{p-1},1 \right)$. If the Cauchy problem
\begin{equation}\label{cauchy_problem_model_operator}
\left\lbrace \begin{array}{l}
Mu = 0 \\ 
u(0,x) = g(x)
\end{array}  \right., \quad (t,x) \in [0,T] \times \R
\end{equation}
is well-posed in $\mathcal{S}^\theta_s(\R)$, then
$$
\max \left\lbrace \frac{1}{(p-1)\theta}, \frac{1}{s} \right\rbrace \geq 1 - \sigma.
$$
\end{proposition}

\begin{remark} From Proposition \ref{proposition_main_result_4} we obtain that if $s > \frac{1}{1-\sigma}$ and $(p-1)\theta > \frac{1}{1-\sigma}$, that is $1-\sigma > \max \left\lbrace \frac{1}{(p-1)\theta}, \frac{1}{s} \right\rbrace$, the Cauchy problem \eqref{cauchy_problem_model_operator} is not well-posed. Since $M$ is of the form \eqref{differential_p_evolution_operator} and satisfies conditions (i) and (ii) of Theorem \ref{theorem_main_result_3}, the combination of Propositions \ref{cauchy_problem_similar_schrodinger} and \ref{proposition_main_result_4} gives that if \eqref{inequality_sharpness} holds, then in general the Cauchy problem \eqref{cauchy_problem_gelfand-shilov} for an operator $P(t,x,D_t,D_x)$ satisfying the assumptions of Theorem \ref{theorem_main_result_3} is not well-posed in $\mathcal{S}^\theta_s(\R)$.
\end{remark}

In view of the considerations above, we devote the rest of this section to prove Proposition \ref{proposition_main_result_4}.

\subsection{Proof of Proposition \ref{proposition_main_result_4}}

The proof needs some preparation and relies on several preliminary results. First of all, we can show that if the Cauchy problem \eqref{cauchy_problem_model_operator} is well-posed in $\mathcal{S}^\theta_s(\R),$ then we can find a sequence of functions $\phi_k \in H^0_{\rho; s, \theta}(\R)$ for some $\rho=(\rho_1,\rho_2) \in \R^2$ with $\rho_j>0, j=1,2,$ such that if $u_k$ satisfies $Mu_k=0$ and $u_k(0,x)=\phi_k(x)$, then 
	$\|u_k(t, \cdot)\|_{L^2(\R)}$ is uniformly bounded with respect to $k \in \N_0$ and $t \in [0,T]$. As a matter of fact, let us consider $\phi \in G^\theta(\R)$ satisfying
	\begin{equation}
		\widehat{\phi}(\xi) = e^{-\rho_0\japxi^{1/\theta}},
	\end{equation}
	for some $\rho_0 > 0$. Notice that $\phi \in \mathcal{S}_1^\theta(\R)$. In order to reach this conclusion, it is sufficient to prove that $\phi(x) = \int e^{i \xi x} \widehat{\phi}(\xi) \dslash \xi$ satisfies $| \phi(x) | \leq Ce^{-c|x|}$, for some positive constants $C$ and $c$, due to Proposition 6.1.7 of \cite{NR11}. Integration by parts leads to
	$$
	x^\beta \phi(x) = (-1)^\beta \int e^{i \xi x} D_\xi^\beta e^{-\rho_0 \japxi^{1/\theta}} \dslash\xi,
	$$
	and Fa{\`a} di Bruno formula together with some factorial inequalities imply the following estimate
	\begin{eqnarray}
		| x^\beta \phi(x) | & \leq & \int \left| e^{i \xi x} D_\xi^\beta e^{-\rho_0 \japxi^{1/\theta}}  \right| \dslash\xi \nonumber \\
		& = & \int \sum_{j=1}^\beta \frac{1}{j!} e^{-\rho_0 \japxi^{1/\theta}} \sum_{\stackrel{\beta_1 + \cdots + \beta_j = \beta}{\beta_\ell \geq 1}} \frac{\beta!}{\beta_1! \cdots \beta_j!} \prod_{\ell = 1}^j \left| \partial_\xi^{\beta_\ell} \left( -\rho_0 \japxi^{1/\theta} \right) \right| \dslash\xi \nonumber \\
		& \leq & C_{\rho_0}^\beta \beta!, \quad \forall x \in \R,\ \beta \in \mathbb{N}_0,
	\end{eqnarray}
	hence, $\phi \in \mathcal{S}_1^\theta(\R) \subset \mathcal{S}_s^\theta(\R)$ for all $s \geq 1$. Then, there exists $\rho = (\rho_1,\rho_2)$, with $\rho_1,\rho_2>0$, such that $\phi \in H_{\rho;s,\theta}^0(\R)$ for all $s \geq 1$.

	Now let us consider a sequence $(\sigma_k)_{k \in \mathbb{N}_0}$ of positive real numbers such that $\sigma_k \to \infty$, as $k \to \infty$. Then we define the sequence of functions $(\phi_k)_{k \in \mathbb{N}_0}$ by setting
	$$
	\phi_k(x) := e^{-\rho_2 4^{1/s} \sigma_k^{\frac{p-1}{s}}} \phi\left(x - 4\sigma_k^{p-1}\right), \quad k \in \mathbb{N}_0.
	$$
	For each $k \in \mathbb{N}_0$, $\phi_k \in H_{\rho;s,\theta}^0(\R)$ and satisfies
	\begin{eqnarray}
		\| \phi_k \|_{H_{\rho;s,\theta}^0}^2 
		& = & \int e^{2\rho_2 \japx^{1/s}} \left| e^{\rho_1\langle D \rangle^{1/\theta}} \phi_k(x) \right|^2 dx \nonumber \\
		& = & \int e^{2\rho_2 \japx^{1/s}} \left| e^{\rho_1\langle D \rangle^{1/\theta}} e^{-\rho_2 4^{1/s} \sigma_k^{\frac{p-1}{s}}} \phi\left(x - 4\sigma_k^{p-1}\right) \right|^2 dx \nonumber \\
		& \leq & \int e^{2\rho_2 \left\langle x + 4\sigma_k^{p-1} \right\rangle^{1/s}} e^{-2\rho_2 4^{1/s} \sigma_k^{\frac{p-1}{s}}} \left| e^{\rho_1\langle D \rangle^{1/\theta}} \phi(x) \right|^2 dx \nonumber \\
		& \leq & \int e^{2\rho_2 \japx^{1/s}} \left| e^{\rho_1 \langle D \rangle^{1/\theta}} \phi(x) \right|^2 dx \nonumber \\
		& = & \| \phi \|_{H_{\rho;s,\theta}^0}^2 = \textrm{constant}, \nonumber
	\end{eqnarray}
	which means that the sequence $\left( \|\phi_k\|_{H_{\rho;s,\theta}^0} \right)_{k \in \mathbb{N}_0}$ is uniformly bounded in $k$. 
	Now, since the Cauchy problem \eqref{cauchy_problem_model_operator} is well-posed in $\mathcal{S}_s^\theta(\R)$, we can consider, for every $k \in \mathbb{N}_0$,
	$$
	u_k \in C^1 \left( [0,T];H_{\tilde{\rho};s,\theta}^0(\R) \right),
	$$
	where $\tilde{\rho} = (\tilde{\rho}_1,\tilde{\rho}_2)$, $\tilde{\rho}_1,\tilde{\rho}_2>0$, the solution of \eqref{cauchy_problem_model_operator} with initial datum $\phi_k$, that is,
	$$
	\left\lbrace \begin{array}{l}
		Mu_k(t,x)=0 \\ 
		u_k(0,x)=\phi_k(x)
	\end{array} \right. .
	$$
	The energy inequality gives the estimate
	\begin{equation}\label{estimate_solution_sequence}
		\| u_k(t,\cdot) \|_{L^2} \leq \| u_k(t,\cdot) \|_{H_{\tilde{\rho};s,\theta}^0} \leq C_{T,\rho} \| \phi_k \|_{H_{\rho;s,\theta}^0} \leq C_{T,\rho} \| \phi \|_{H_{\rho;s,\theta}^0}.
	\end{equation}
	From the above inequality, it can be concluded that the sequence $\left( \| u_k(t) \|_{L^2} \right)_{k \in \mathbb{N}_0}$ is uniformly bounded with respect to $k \in \mathbb{N}_0$ and $t \in [0,T]$.

In order to prove Proposition \ref{proposition_main_result_4}, following an idea used in previous papers, cf. \cite{Arias_GS, AACpevolGevreynec, CRnec, Ichinose1, dreher}, we assume the Cauchy problem \eqref{cauchy_problem_model_operator} to be well-posed in $\mathcal{S}_s^\theta(\R)$ and define a suitable energy $E_k(t)$ associated to the solution $u_k$ and prove for it two estimates one from above and one from below which contradict each other if \eqref{inequality_sharpness} holds. With this purpose, let us consider a Gevrey cut-off function $h \in G_0^{\theta_h}(\R)$, for $\theta_h>1$, such that $h(x) = 1$ for $|x|\leq\frac{1}{2}$ and $h(x)=0$ for $|x|\geq1$; besides, following \cite{dreher}, we assume that its Fourier transform satisfies $\widehat{h}(0)>0$ and $\widehat{h}(\xi) \geq 0$ for all $\xi \in \R$. By using the sequence $(\sigma_k)_{k \in \mathbb{N}_0}$, let us define the sequence of symbols $\{ w_{k} = w_{k}(x,\xi) \}_{k}$ with
\begin{equation}\label{sequence_of_symbols}
w_k(x,\xi) = h \left( \frac{x - 4\sigma_k^{p-1}}{\sigma_k^{p-1}} \right) h \left( \frac{\xi - \sigma_k}{\frac{1}{4}\sigma_k} \right).
\end{equation}

%

\begin{remark}\label{remark_comparable}
On the support of $w_k$, $x$ is comparable with $\sigma_k^{p-1}$ and $\xi$ is comparable with $\sigma_k$, because $x \in \text{\textup{supp}} \ w_k(\cdot,\xi)$ implies that
$$
\left| \frac{x-4\sigma_k^{p-1}}{\sigma_k^{p-1}} \right| \leq 1 \quad \Longleftrightarrow \quad 3\sigma_k^{p-1} \leq x \leq 5\sigma_k^{p-1},
$$
and $\xi \in \text{\textup{supp}} \ w_k(x,\cdot)$ implies that
$$
\left| \frac{\xi-\sigma_k}{\frac{1}{4}\sigma_k} \right| \leq 1 \quad \Longleftrightarrow \quad \frac{3}{4}\sigma_k \leq \xi \leq \frac{5}{4}\sigma_k,
$$
for each $k \in \mathbb{N}_0$.
\end{remark}

For some $\lambda \in (0,1)$ to be chosen later, let $\theta_1>1$ such that $\theta_h \leq \theta_1$ and, for each $k \in \mathbb{N}_0$, let
\begin{equation}\label{equation_N_k}
N_k := \left\lfloor \sigma_k^{\lambda/\theta_1} \right\rfloor = \max \left\lbrace \alpha \in \mathbb{N}_0: \ \alpha \leq \sigma_k^{\lambda/\theta_1} \right\rbrace.
\end{equation}
For each $k$, we consider the \textit{energy} $E_k(t)$ given by
\begin{equation}\label{energy_E_k(t)}
E_k(t) = \sum_{\alpha \leq N_k, \beta \leq N_k} \frac{1}{(\alpha!\beta!)^{\theta_1}} \left\| w_k^{(\alpha\beta)}(x,D) u_k(t,x) \right\|_{L^2} = \sum_{\alpha \leq N_k, \beta \leq N_k} E_{k,\alpha,\beta}(t),
\end{equation}
where
$$
w_k^{(\alpha\beta)}(x,\xi) = h^{(\alpha)} \left( \frac{x - 4\sigma_k^{p-1}}{\sigma_k^{p-1}} \right) h^{(\beta)} \left( \frac{\xi - \sigma_k}{\frac{1}{4}\sigma_k} \right).
$$
From Remark \ref{remark_comparable} by a simple computation it can be easily established the next lemma which gives us an estimate for the norms of $w_k$. The proof is left to the reader.

\begin{lemma}\label{lemma_estimates_w_k}
Let $\alpha,\beta,\gamma,\delta,\ell \in \mathbb{N}_0$. Then $w_k^{(\alpha\beta)} \in S_{0,0}^0(\R^2)$ and satisfies the estimate
$$
| \partialx^\delta \partialxi^\gamma w_k^{(\alpha\beta)}(x,\xi) |_\ell^{(0)} \leq C^{\alpha+\beta+\gamma+\delta+\ell+1} (\alpha!\beta!\gamma!\delta!\ell!^2)^{\theta_h} \sigma_k^{-\gamma} \sigma_k^{-\delta(p-1)},
$$
for some constant $C>0$ which does not depend on $k,\alpha,\beta,\gamma$ and $\delta$.
\end{lemma}

Using the previous result we can easily estimate from above the energy $E_k(t)$.

\begin{lemma}\label{proposition_estimates_E_k_above}
Suppose that the Cauchy problem \eqref{cauchy_problem_model_operator} is well-posed in $\mathcal{S}_s^\theta(\R)$. Then there exists $C>0$ such that, for all $t \in [0,T]$ and $k \in \mathbb{N}_0$:
\begin{equation}\label{E_kalphabeta_estimate}
E_{k,\alpha,\beta}(t) \leq C^{\alpha+\beta+1} (\alpha!\beta!)^{\theta_h-\theta_1}
\end{equation}
and
\begin{equation}\label{E_k_less_than_constant}
E_k(t) \leq C.
\end{equation}
\end{lemma}

\begin{proof}
Since $w_k^{(\alpha\beta)} \in S_{0,0}^0(\R^2)$, Proposition \ref{CVthm} of the Appendix together with \eqref{estimate_solution_sequence} and Lemma \ref{lemma_estimates_w_k} implies that
\begin{eqnarray}
\left\| w_k^{(\alpha\beta)}(x,D) u_k(t) \right\|_{L^2} & \leq & C \| u_k(t) \|_{L^2} \max_{\alpha,\beta \leq 2} \sup_{(x,\xi)\in \R^2} | \doublepartial w_k(x,\xi) | \nonumber \\
& \leq & C^{\alpha+\beta+1} (\alpha!\beta!)^{\theta_h}, \nonumber
\end{eqnarray}
hence, by the definition of $E_{k,\alpha,\beta}(t)$, we have
$$
E_{k,\alpha,\beta}(t) = \frac{1}{(\alpha!\beta!)^{\theta_1}} \left\| w_k^{(\alpha\beta)}(x,D) u_k(t) \right\|_{L^2} \leq C^{\alpha+\beta+1}(\alpha!\beta!)^{\theta_h-\theta_1}.
$$
Since $\theta_1>\theta_h$, we finally obtain \eqref{E_k_less_than_constant}, that is,
$$
E_k(t) = \sum_{\alpha \leq N_k,\beta \leq N_k} E_{k,\alpha,\beta}(t) \leq C \sum_{\alpha,\beta \in \mathbb{N}_{0}} C^{\alpha+\beta} (\alpha!\beta!)^{\theta_h-\theta_1} = \textrm{constant}.
$$
\end{proof}

The estimate from below of $E_k(t)$ is more involved and it will be obtained by Gronwall's lemma after proving a suitable estimate from below of $\partial_t E_k(t)$. This is the content of the next lemma.

\begin{proposition}\label{proposition_estimates_E_k_below}
Suppose that the Cauchy problem \eqref{cauchy_problem_model_operator} is well-posed in $\mathcal{S}_s^\theta(\R)$. Then there exist positive constants $C$, $c$, $c_1$ such that, for all $t \in [0,T]$ and all $k$ sufficiently large, the inequality 
$$
\partial_t E_k(t) \geq \left( c_1 \sigma_k^{(p-1)(1-\sigma)} - C \sum_{\ell=1}^p \frac{\sigma_k^{\ell\lambda}}{\sigma_k^{p(\ell-1)}} \right) E_k(t) - C^{N_k+1} \sigma_k^{C-cN_k},
$$
holds.
\end{proposition}

Since the proof of Proposition \ref{proposition_estimates_E_k_below} is long and requires several preliminary steps, we prefer to devote a separate subsection to it, see Subsection \ref{longproof}. In the next lines, we give Proposition \ref{proposition_estimates_E_k_below} as true and derive from it and from Lemma \ref{proposition_estimates_E_k_above} the conclusion of Proposition \ref{proposition_main_result_4}. 

\begin{proof}[Proof of Proposition \ref{proposition_main_result_4}]
First of all, we set
$$
A_k := c_1 \sigma_k^{(p-1)(1-\sigma)} - C \sum_{\ell=1}^p \frac{\sigma_k^{\ell\lambda}}{\sigma_k^{p(\ell-1)}}, \quad R_k := C^{N_k+1} \sigma_k^{C-cN_k}.
$$
By Proposition \ref{proposition_estimates_E_k_below}, we have that
\begin{equation}\label{estimate_E_k_from_below}
\partial_t E_k(t) \geq A_k E_k(t) - R_k.
\end{equation}
Picking $\lambda < \min \{ (p-1)(1-\sigma), 1 \}$, we notice that
$$
(\ell-1)\underbrace{(\lambda-p)}_{<0} + \lambda < \lambda < (p-1)(1-\sigma) \quad \Rightarrow \quad \ell\lambda-p(\ell-1)<(p-1)(1-\sigma),
$$
so the leading term in $A_k$ is the first one, hence, for $k$ sufficiently large we have
\begin{equation}\label{estimate_A_k_from_below}
A_k \geq \frac{c_1}{2} \sigma_k^{(p-1)(1-\sigma)}.
\end{equation}
From now on, let us consider $k$ sufficiently large. Then, by using Gronwall's inequality, if follows from \eqref{estimate_E_k_from_below} that
$$
E_k(t) \geq e^{A_k t} \left( E_k(0) - R_k \int_0^t e^{-A_k \tau} d\tau \right).
$$
Then, for any $T^\ast \in [0,T]$, by using \eqref{estimate_A_k_from_below}, the above inequality turns into
\begin{equation}\label{estimate_E_k(T*)_from_below}
E_k(T^\ast) \geq e^{T^\ast \frac{c_1}{2} \sigma_k^{(p-1)(1-\sigma)}} \left( E_k(0) - T^\ast R_k \right).
\end{equation}
The next step is to estimate $R_k$ and $E_k(0)$. Since $N_k = \left\lfloor \sigma_k^{\lambda/\theta_1} \right\rfloor$, $R_k$ can be estimated as
\begin{equation}\label{estimate_R_k_from_above}
R_k \leq C e^{-c \sigma_k^{\lambda/\theta_1}}.
\end{equation}
To obtain an estimate from below for $E_k(0)$, we notice that, by definition of $E_k(t)$, $w_k(x,\xi)$ and $\phi_k$, we get
\begin{eqnarray}
E_k(0) & \geq & \| w_k(x,D) \phi_k \|_{L^2(\R_x)} \nonumber \\
& = & \left\| h \left(\frac{x-4\sigma_k^{p-1}}{\sigma_k^{p-1}} \right) h \left( \frac{D-\sigma_k}{\frac{1}{4}\sigma_k} \right) \phi_k \right\|_{L^2(\R_x)} \nonumber \\
& = & e^{-\rho_2 4^{1/s} \sigma_k^{(p-1)/s}} \left\| h \left(\frac{x-4\sigma_k^{p-1}}{\sigma_k^{p-1}} \right) h \left( \frac{D-\sigma_k}{\frac{1}{4}\sigma_k} \right) \phi(x-4\sigma_k^{p-1}) \right\|_{L^2(\R_x)} \nonumber \\
& = & e^{-\rho_2 4^{1/s} \sigma_k^{(p-1)/s}} \left\| \mathscr{F} \left[ h\left(\frac{x-4\sigma_k^{p-1}}{\sigma_k^{p-1}} \right) \right](\xi) \ast h\left( \frac{D-\sigma_k}{\frac{1}{4}\sigma_k} \right) e^{-4i\sigma_k^{p-1} \xi} \widehat{\phi}(\xi) \right\|_{L^2(\R_\xi)} \nonumber \\
& = & e^{-\rho_2 4^{1/s} \sigma_k^{(p-1)/s}} \sigma_k^{p-1} \left\| e^{-4i\sigma_k^{p-1}\xi} \widehat{h}\left(\sigma_k^{p-1}\xi\right) \ast h \left(\frac{D-\sigma_k}{\frac{1}{4}\sigma_k} \right) e^{-4i\sigma_k^{p-1}\xi} \widehat{\phi}(\xi) \right\|_{L^2(\R_\xi)}, \nonumber
\end{eqnarray}
where $h\left( \frac{D-\sigma_k}{\frac{1}{4}\sigma_k} \right)$ denotes the pseudodifferential operator with symbol $h\left( \frac{\xi-\sigma_k}{\frac{1}{4}\sigma_k} \right)$. Hence,
$$
E_k^2(0) \geq e^{-2\rho_2 4^{1/s} \sigma_k^{(p-1)/s}} \sigma_k^{2(p-1)} \int_{\R_\xi} \left| \int_{\R_\eta} \widehat{h}\left( \sigma_k^{p-1} (\xi - \eta) \right) h \left( \frac{\eta-\sigma_k}{\frac{1}{4}\sigma_k} \right) \widehat{\phi}(\eta) d\eta \right|^2 d\xi.
$$
Since $\widehat{h}(0) > 0$ and $\widehat{h}(\xi) \geq 0$, for all $\xi \in \R$, it will be possible to obtain estimates from below to $E_k^2(0)$ performing a restriction in the integration domain. Set
$$
G_{1,k} := \left[ \frac{7}{8}\sigma_k, \frac{7}{8}\sigma_k + \sigma_k^{-p} \right] \quad \text{and} \quad G_{2,k} := \left[ \frac{7}{8}\sigma_k - \sigma_k^{-p}, \frac{7}{8}\sigma_k + \sigma_k^{-p} \right].
$$
Notice that, if $\eta \in G_{1,k}$ then $| \eta - \sigma_k | \leq \frac{\sigma_k}{8}$, because
$$
\eta \in G_{1,k} \ \Rightarrow \ \frac{7}{8}\sigma_k \leq \eta \leq \frac{7}{8}\sigma_k + \sigma_k^{-p} \ \Rightarrow \ -\frac{\sigma_k}{8} \leq \eta - \sigma_k \leq -\frac{\sigma_k}{8} + \sigma_k^{-p}.
$$
Also, if $\eta \in G_{1,k}$ and $\xi \in G_{2,k}$, we notice that
$$
\sigma_k^{p-1} | \xi - \eta | \leq 2 \sigma_k^{-1},
$$
because $\eta \in G_{1,k}$ implies that
\begin{equation}\label{G_1}
-\sigma_k^{-p} - \frac{7}{8}\sigma_k \leq -\eta \leq - \frac{7}{8}\sigma_k,
\end{equation}
$\xi \in G_{2,k}$ implies that
\begin{equation}\label{G_2}
\frac{7}{8}\sigma_k - \sigma_k^{-p} \leq \xi \leq \frac{7}{8}\sigma_k + \sigma_k^{-p},
\end{equation}
and from \eqref{G_1} and \eqref{G_2} we get
$$
-2 \sigma_k^{-p} \leq \xi - \eta \leq \sigma_k^{-p} \leq 2 \sigma_k^{-p} \ \Rightarrow \ | \xi - \eta | \leq 2 \sigma_k^{-p} \ \Rightarrow \ \sigma_k^{p-1} | \xi - \eta | \leq 2 \sigma_k^{-1}.
$$
If we pick $(\xi,\eta) \in G_{2,k} \times G_{1,k}$, then $\sigma_k^{p-1} ( \xi - \eta )$ is close to zero for $k$ large enough, hence, by the choices of $\widehat{h}(0) > 0$ and $\widehat{h}(\xi) \geq 0$, there exists a constant $C > 0$ such that
$$
\widehat{h} \left( \sigma_k^{p-1} ( \xi - \eta ) \right) > C.
$$
Furthermore, if $\eta \in G_{1,k}$ then $h \left( \frac{\eta - \sigma_k}{\frac{1}{4}\sigma_k} \right) = 1$. Finally, since $\eta$ is comparable with $\sigma_k$ in $G_{1,k}$, it follows from $\widehat{\phi}(\eta) = e^{-\rho_0 \langle \eta \rangle^{1/\theta}}$ that
$$
\widehat{\phi}(\eta) \geq e^{-c_{\rho_0} \sigma_k^{1/\theta}},
$$
for a positive constant $c_{\rho_0}$ depending on $\rho_0$. Finally, we obtain
\begin{eqnarray}
E_k^2(0) & \geq & e^{-2 \rho_2 4^{1/s} \sigma_k^{(p-1)/s}} \sigma_k^{2(p-1)} \int_{G_{2,k}} \left| \int_{G_{1,k}} C e^{-c_{\rho_0} \sigma_k^{1/\theta}} d\eta \right|^2 d\xi \nonumber \\
& = & C e^{-2 \rho_2 4^{1/s} \sigma_k^{(p-1)/s}} \sigma_k^{-(p+2)} e^{-2c_{\rho_0} \sigma_k^{1/\theta}} \nonumber
\end{eqnarray}
which implies that
\begin{equation}\label{estimate_E_k(0)_from_below}
E_k(0) \geq C \sigma_k^{-(p+2)/2} \exp \left( -\rho_2 4^{1/s} \sigma_k^{(p-1)/s}\right) \exp \left( -c_{\rho_0} \sigma_k^{1/\theta} \right) \geq C \exp \left( -\tilde{c}_{\rho_0,\rho_2} \sigma_k ^{\max\left\lbrace \frac{p-1}{s}, \frac{1}{\theta} \right\rbrace} \right).
\end{equation}
From \eqref{estimate_E_k(T*)_from_below}, \eqref{estimate_R_k_from_above} and \eqref{estimate_E_k(0)_from_below} we get
\begin{eqnarray}\label{estimate_E_k(T*)_from_below_2}
E_k(T^\ast) & \geq & C \exp \left( T^\ast \frac{c_1}{2} \sigma_k^{(p-1)(1-\sigma)} \right) \left[ \sigma_k^{-\frac{p+2}{2}} \exp \left( -\rho_2 4^\frac{1}{s} \sigma_k^\frac{p-1}{s} - c_{\rho_0} \sigma_k^\frac{1}{\theta} \right) - T^\ast \exp \left( -c \sigma_k^\frac{\lambda}{\theta_1} \right) \right] \nonumber \\
& \geq & C_1 \exp \left( \frac{c_1}{2} T^\ast \sigma_k^{(p-1)(1-\sigma)} \right) \left[ \exp \left( - \tilde{c}_{\rho_0,\rho_2} \sigma_k^{\max \left\lbrace \frac{p-1}{s},\frac{1}{\theta} \right\rbrace} \right) - T^\ast \exp \left( -c\sigma_k^\frac{\lambda}{\theta_1} \right) \right]
\end{eqnarray}
for all $T^\ast \in (0,T]$, once $\lambda < (p-1)(1-\sigma)$ and $k$ is sufficiently large. Assume by contradiction that $\max\left\lbrace \frac{1}{(p-1)\theta},\frac{1}{s} \right\rbrace < 1 - \sigma$, which is equivalent to $\max\left\lbrace \frac{1}{\theta},\frac{p-1}{s} \right\rbrace < (p-1)(1-\sigma)$, and take $\lambda < (p-1)(1-\sigma)$. Then, we can pick $\theta_1$ very close to $1$ such that
$$
\frac{\lambda}{\theta_1} > \max\left\lbrace \frac{1}{\theta},\frac{p-1}{s} \right\rbrace.
$$
It follows from \eqref{estimate_E_k(T*)_from_below_2} that
$$
E_k(T^\ast) \geq C_2 \exp \left( \frac{\tilde{c} T^\ast}{2} \sigma_k^{(p-1)(1-\sigma)} \right) \exp \left( -c^{\prime\prime} \sigma_k^{\max\left\lbrace \frac{1}{\theta},\frac{p-1}{s} \right\rbrace} \right) \longrightarrow \infty, \ \text{as} \ k \to \infty,
$$
since $(p-1)(1-\sigma) > \max\left\lbrace \frac{1}{\theta},\frac{p-1}{s} \right\rbrace$, which gives us a contradiction, due to Lemma \ref{proposition_estimates_E_k_above}.

\end{proof}

\subsection{Proof of Proposition \ref{proposition_estimates_E_k_below}} \label{longproof}
Let us set
$$
v_k^{(\alpha\beta)}(t,x) := w_k^{(\alpha\beta)}(x,D) u_k(t,x).
$$
By denoting $[H,K] = HK - KH$, for $H,K$ operators, we can write
\begin{eqnarray}
Mv_k^{(\alpha\beta)} & = & Mw_k^{(\alpha\beta)}u_k \nonumber \\
& = & w_k^{(\alpha\beta)}\underbrace{Mu_k}_{=0} + [M,w_k^{(\alpha\beta)}]u_k \nonumber \\
& = & [M,w_k^{(\alpha\beta)}]u_k =: f_k^{(\alpha\beta)}. \nonumber
\end{eqnarray}
In order to obtain an estimate from below for $\partial_t E_k$, we notice that
\begin{equation}\label{norm_partial_norm}
\| v_k^{(\alpha\beta)}(t) \|_{L^2(\R)} \partial_t \| v_k^{(\alpha\beta)}(t) \|_{L^2(\R)}  = \mathbf{Re} \left( \partial_t v_k^{(\alpha\beta)}(t), v_k^{(\alpha\beta)}(t) \right)_{L^2(\R)}.
\end{equation}
The definition of $M$ implies that $\partial_t = iM - iD_x^p + \japx^{-\sigma}D_x^{p-1}$. Hence, we can rewrite \eqref{norm_partial_norm} and estimate from below
\begin{eqnarray}\label{norm_partial_norm_1}
& & \| v_k^{(\alpha\beta)}(t) \|_{L^2(\R)} \partial_t \| v_k^{(\alpha\beta)}(t) \|_{L^2(\R)} \nonumber \\
& = & \mathbf{Re} \left( iMv_k^{(\alpha\beta)}(t),v_k^{(\alpha\beta)}(t) \right)_{L^2(\R)} - \mathbf{Re} \underbrace{\left( iD_x^pv_k^{(\alpha\beta)}(t),v_k^{(\alpha\beta)}(t) \right)_{L^2(\R)}}_{\text{purely imaginary by Parseval's identity}} \nonumber \\
& + & \mathbf{Re} \left( \japx^{-\sigma} D_x^{p-1} v_k^{(\alpha\beta)}(t),v_k^{(\alpha\beta)}(t) \right)_{L^2(\R)} \nonumber \\
& \geq & - \| f_k^{(\alpha\beta)} (t) \|_{L^2(\R)} \| v_k^{(\alpha\beta)}(t) \|_{L^2(\R)} + \mathbf{Re} \left( \japx^{-\sigma} D_x^{p-1} v_k^{(\alpha\beta)}(t),v_k^{(\alpha\beta)}(t) \right)_{L^2(\R)}.
\end{eqnarray}
Consequently, we need to estimate the following terms:
\begin{itemize}
\item[1.] $\mathbf{Re} \left( \japx^{-\sigma} D_x^{p-1} v_k^{(\alpha\beta)}(t),v_k^{(\alpha\beta)}(t) \right)_{L^2(\R)}$ from below,
\item[2.] $\| f_k^{(\alpha\beta)} (t) \|_{L^2(\R)}$ from above.
\end{itemize}

In the proof of the two estimates above we will use the following lemma whose proof follows readily the same argument of the proof of Lemma 2 in \cite{AACpevolGevreynec}, the only difference is that here we require well-posedness in $\mathcal{S}^\theta_s(\R)$. We omit the proof and refer the reader to \cite{AACpevolGevreynec} for the sake of brevity.

\begin{lemma}\label{lemma_derivatives_of_v_k^alphabeta}
	If the Cauchy problem \eqref{cauchy_problem_model_operator} is well-posed in $\mathcal{S}^\theta_s(\R)$, then for all $r, N \in \mathbb{N}$, the following estimate holds:
	$$
	\| D_x^r v_k^{(\alpha\beta)}(t) \|_{L^2(\R)} \leq C \sigma_k^r \| v_k^{(\alpha\beta)}(t) \|_{L^2(\R)} + C^{\alpha+\beta+N+1} (\alpha!\beta!)^{\theta_h} N!^{2\theta_h-1} \sigma_k^{r-N},
	$$
	for some positive constant $C$ which does not depend on $k$.
\end{lemma}

\textbf{1. Estimate from below of $\mathbf{Re} \left( \japx^{-\sigma} D_x^{p-1} v_k^{(\alpha\beta)}(t),v_k^{(\alpha\beta)}(t) \right)_{L^2(\R)}$.}

Let $\chi_k$ and $\psi_k$ cut-off functions defined by
\begin{equation}\label{cut-off_functions}
\chi_k(\xi) = h \left( \frac{\xi - \sigma_k}{\frac{3}{4}\sigma_k} \right) \quad \text{and} \quad \psi_k(x) = h \left( \frac{x-4\sigma_k^{p-1}}{3\sigma_k^{p-1}} \right),
\end{equation}
respectively. Recalling the definition of $h$, we notice that in the support of $\psi_k(x)\chi_k(\xi)$ we have
$$
| \xi - \sigma_k | \leq \frac{3}{4}\sigma_k \quad \Leftrightarrow \quad \frac{\sigma_k}{4} \leq \xi \leq \frac{7\sigma_k}{4}
$$
and
$$
| x - 4\sigma_k^{p-1} | \leq 3\sigma_k^{p-1} \quad \Leftrightarrow \quad \sigma_k^{p-1} \leq x \leq 7\sigma_k^{p-1},
$$
for all $k \in \mathbb{N}_0$. From these inequalities, it follows that if $(x,\xi) \in \text{supp} \{\psi_k(x) \chi_k(\xi)\}$, then
$$
\xi^{p-1} \geq \frac{\sigma_k^{p-1}}{4^{p-1}} \quad \text{and} \quad \japx^{-\sigma} \geq 7^{-\sigma} \langle \sigma_k^{p-1} \rangle^{-\sigma},
$$
hence,
$$
\xi^{p-1} \japx^{-\sigma} \geq \frac{7^{-\sigma}}{4^{p-1}} \langle \sigma_k^{p-1} \rangle^{-\sigma} \sigma_k^{p-1}.
$$
By setting $\Theta:=\frac{7^{-\sigma}}{4^{p-1}}$, let us consider a decomposition of the symbol of $\japx^{-\sigma}D_x^{p-1}$ in the following way:
\begin{eqnarray}
\japx^{-\sigma}\xi^{p-1} & = &  \underbrace{\Theta \langle \sigma_k^{p-1} \rangle^{-\sigma} \sigma_k^{p-1}}_{=:\mathtt{I}_{1,k}} + \underbrace{\left( \japx^{-\sigma}\xi^{p-1} - \Theta \langle \sigma_k^{p-1} \rangle^{-\sigma} \sigma_k^{p-1} \right) \psi_k(x) \chi_k(\xi)}_{=:\mathtt{I}_{2,k}(x,\xi)} \nonumber \\
& + & \underbrace{\left( \japx^{-\sigma} \xi^{p-1} - \Theta \langle \sigma_k^{p-1} \rangle^{-\sigma} \sigma_k^{p-1} \right) \left( 1 - \psi_k(x)\chi_k(\xi) \right)}_{=:\mathtt{I}_{3,k}(x,\xi)} \nonumber \\
& = & \mathtt{I}_{1,k} + \mathtt{I}_{2,k}(x,\xi) + \mathtt{I}_{3,k}(x,\xi). \nonumber
\end{eqnarray}
In the following, we shall estimate each $\mathbf{Re} \left( \mathtt{I}_{\ell,k}(x,D) v_k^{(\alpha\beta)}(t), v_k^{(\alpha\beta)}(t) \right)_{L^2(\R)}$, $\ell \in \{1,2,3\}$.

\begin{itemize}

\item \textbf{Estimate of $\mathbf{Re} \left( I_{1,k} v_k^{(\alpha\beta)}(t),v_k^{(\alpha\beta}(t) \right)_{L^2(\R)}$.} Since $\mathtt{I}_{1,k} = \Theta \langle \sigma_k^{p-1} \rangle^{-\sigma} \sigma_k^{p-1}$, we get
\begin{eqnarray}\label{estimate_from_below_I_1,k}
\mathbf{Re} \left( \mathtt{I}_{1,k} v_k^{(\alpha\beta)} (t), v_k^{(\alpha\beta)}(t) \right)_{L^2(\R)} & = & \mathbf{Re} \left( \Theta \langle \sigma_k^{p-1} \rangle^{-\sigma} \sigma_k^{p-1} v_k^{(\alpha\beta)}(t), v_k^{\alpha\beta}(t) \right)_{L^2(\R)} \nonumber \\
& = & \Theta \langle \sigma_k^{p-1} \rangle^{-\sigma} \sigma_k^{p-1} \| v_k^{(\alpha\beta)}(t) \|_{L^2(\R)}^2 \nonumber \\
& \geq & 2^{-\frac{\sigma}{2}} \Theta \sigma_k^{(p-1)(1-\sigma)} \| v_k^{(\alpha\beta)}(t) \|_{L^2(\R)}^2.
\end{eqnarray}

\item \textbf{Estimate of $\mathbf{Re} \left( \mathtt{I}_{2,k}(x,D) v_k^{(\alpha\beta)}(t), v_k^{(\alpha\beta)}(t) \right)_{L^2(\R)}$.} Just recall that
$$
\mathtt{I}_{2,k}(x,\xi) = \left( \japx^{-\sigma} \xi^{p-1} - \Theta \langle \sigma_k^{p-1} \rangle^{-\sigma} \sigma_k^{p-1} \right) \psi_k(x) \chi_k(\xi),
$$
and this symbol belongs to $\mathbf{SG}^{p-1,-\sigma}(\R^2)$, cf. Definition \ref{SGsymbols}, with uniform estimates with respect to $k$. Indeed, we have that
\begin{eqnarray}
| \partialxi^\gamma \partialx^\nu \mathtt{I}_{2,k}(x,\xi) | & = & \left| \partialxi^\gamma\partialx^\nu \left[ \left( \japx^{-\sigma} \xi^{p-1} - \Theta \langle \sigma_k^{p-1} \rangle^{-\sigma} \sigma_k^{p-1} \right) \psi_k(x) \chi_k(\xi) \right] \right| \nonumber \\
& \leq & \sum_{\stackrel{\gamma_1+\gamma_2=\gamma}{\nu_1+\nu_2=\nu}} \frac{\gamma!\nu!}{\gamma_1!\gamma_2!\nu_1!\nu_2!} \left| \partialxi^{\gamma_1} \partialx^{\nu_1} \left( \japx^{-\sigma} \xi^{p-1} - \Theta \langle \sigma_k^{p-1} \rangle \sigma_k^{p-1} \right) \right| \nonumber \\
& \times &| \partialx^{\nu_2} \psi_k(x) | | \partialxi^{\gamma_2} \chi_k(\xi) |. \nonumber
\end{eqnarray}
Since $h \in G_0^{\theta_h}(\R)$, $\psi_k$ and $\chi_k$ satisfy, respectively,
$$
| \partialx^{\nu_2} \psi_k(x) | = \left| \partialx^{\nu_2} h \left( \frac{x-4\sigma_k^{p-1}}{3\sigma_k^{p-1}} \right) \right| \leq C^{\nu_2 + 1} \nu_2!^{\theta_h} \sigma_k^{-(p-1)\nu_2}
$$
and
$$
| \partialxi^{\gamma_2} \chi_k(\xi) | = \left| \partialxi^{\gamma_2} h \left( \frac{\xi - \sigma_k}{\frac{3}{4}\sigma_k} \right) \right| \leq C^{\gamma_2 + 1} \gamma_2!^{\theta_h} \sigma_k^{-\gamma_2}.
$$
Now, by using the above inequalities and the fact that $x$ is comparable with $\sigma_k^{p-1}$ and $\xi$ is comparable with $\sigma_k$ on the support of $\psi_k(x)\chi_k(\xi)$, we can compute and estimate
\begin{eqnarray}
& & | \partialxi^\gamma \partialx^\nu \mathtt{I}_{2,k}(x,\xi) | \nonumber \\
& \leq & \sum_{\stackrel{\gamma_1+\gamma_2=\gamma}{\nu_1+\nu_2=\nu}} \frac{\gamma!\nu!}{\gamma_1!\gamma_2!\nu_1!\nu_2!} C^{\gamma_1 + \nu_1 + 1} \gamma_1! \nu_1! \japx^{-\sigma-\nu_1} \japxi^{p - 1 - \gamma_1} C^{\gamma_1 + \nu_2 + 1} \gamma_2!^{\theta_h} \nu_2!^{\theta_h} \sigma_k^{-\gamma_2} \sigma_k^{-\nu_2(p-1)} \nonumber \\
& \leq & C^{\gamma + \nu + 1} (\gamma!\nu!)^{\theta_h} \japxi^{p - 1 - \gamma} \japx^{-\sigma - \nu}. \nonumber
\end{eqnarray}
Moreover, by the choice of $\Theta$, we have that $\mathtt{I}_{2,k}(x,\xi) \geq 0$, for all $x,\xi \in \R$. Then it follows from Proposition \ref{theorem_sharp_garding} that
$$
\mathtt{I}_{2,k}(x,D) = \mathtt{p}_{2,k}(x,D) + \mathtt{r}_{2,k}(x,D),
$$
where $\mathtt{p}_{2,k}(x,D)$ is a nonnegative operator and $\mathtt{r}_{2,k} \in \mathbf{SG}^{p-2,-\sigma-1}(\R^2)$. Once the semi-norms of $\mathtt{I}_{2,k}$ are uniformly bounded with respect to $k$, the same holds to the semi-norms of $\mathtt{r}_{2,k}$. In this way, we have that
\begin{eqnarray}
\mathbf{Re} \left( \mathtt{I}_{2,k}(x,D) v_k^{(\alpha\beta)}(t), v_k^{(\alpha\beta)}(t) \right)_{L^2(\R)} & \geq & \mathbf{Re} \left( \mathtt{r}_{2,k}(x,D) v_k^{(\alpha\beta)}(t), v_k^{(\alpha\beta)}(t) \right)_{L^2(\R)} \nonumber \\
& = & \mathbf{Re} ( \underbrace{\mathtt{r}_{2,k}\japx^{\sigma+1}}_{\text{order} \ p-2} \japx^{-\sigma-1} v_k^{(\alpha\beta)}(t), v_k^{(\alpha\beta)}(t) )_{L^2(\R)} \nonumber \\
& \geq & -C \| \japx^{-\sigma-1} v_k^{(\alpha\beta)}(t) \|_{H^{p-2}(\R)} \| v_k^{(\alpha\beta)}(t) \|_{L^2(\R)} \nonumber \\
& \geq & -C \| v_k^{(\alpha\beta)}(t) \|_{L^2(\R)} \sum_{\ell = 0}^{p-2} \left\| D_x^\ell \left( \japx^{-\sigma-1} v_k^{(\alpha\beta)}(t) \right) \right\|_{L^2(\R)}. \nonumber
\end{eqnarray}
To deal with the terms $\left\| D_x^\ell \left( \japx^{-\sigma-1} v_k^{(\alpha\beta)}(t) \right) \right\|_{L^2(\R)}$ we need to use the Leibniz formula and write
$$
D_x^\ell \left( \japx^{-\sigma-1} v_k^{(\alpha\beta)}(t) \right) = \sum_{\ell^\prime=0}^\ell \binom{\ell}{\ell^\prime} D_x^{\ell^\prime} \japx^{-\sigma-1} D_x^{\ell-\ell^\prime} v_k^{(\alpha\beta)}(t).
$$
On the support of $D_x^{\ell-\ell^\prime}v_k^{(\alpha\beta)}(t)$, $x$ is comparable with $\sigma_k^{p-1}$. Hence,
$$
\left\| D_x^\ell \left( \japx^{-\sigma-1} v_k^{(\alpha\beta)}(t) \right) \right\|_{L^2(\R)}  \leq C^{\ell+1} \sigma_k^{-(p-1)(\sigma+1)} \sum_{\ell^\prime=0}^\ell \frac{\ell!}{(\ell-\ell^\prime)!} \|D_x^{\ell-\ell^\prime} v_k^{(\alpha\beta)}(t) \|_{L^2(\R)}.
$$
Applying Lemma \ref{lemma_derivatives_of_v_k^alphabeta} for $N=N_k$, the following estimate can be achieved:
\begin{eqnarray}
\left\| D_x^\ell \left( \japx^{-\sigma-1} v_k^{(\alpha\beta)}(t) \right) \right\|_{L^2(\R)} & \leq & C \sigma_k^{-(p-1)(\sigma+1)+\ell} \|v_k^{(\alpha\beta)}\|_{L^2(\R)} \nonumber \\
& + & C^{\alpha+\beta+N_k+1} (\alpha!\beta!)^{\theta_h} N_k!^{2\theta_h-1} \sigma_k^{-(p-1)(\sigma+1)} \sigma_k^{\ell-N_k}. \nonumber
\end{eqnarray}
Therefore, we conclude that
\begin{multline}\label{estimate_from_below_I_2,k}
\mathbf{Re} \left( \mathtt{I}_{2,k}(x,D)v_k^{(\alpha\beta)}(t),v_k^{(\alpha\beta)}(t) \right)_{L^2(\R)}  \geq  -C \sigma_k^{-(p-1)\sigma - 1} \| v_k^{(\alpha\beta)}(t) \|_{L^2(\R)}^2  \\
 -  C^{\alpha+\beta+N_k+1} (\alpha!\beta!)^{\theta_h} N_k!^{2\theta_h-1} \sigma_k^{-(p-1)\sigma-N_k-1} \| v_k^{(\alpha\beta)}(t) \|_{L^2(\R)}. 
\end{multline}

\item \textbf{Estimate of $\mathbf{Re} \left( \mathtt{I}_{3,k}(x,D)v_k^{(\alpha\beta)}(t),v_k^{(\alpha\beta)}(t) \right)_{L^2(\R)}$.} Recalling that
$$
\mathtt{I}_{3,k}(x,\xi) = \left( \japx^{-\sigma} \xi^{p-1} - \Theta \langle \sigma_k^{p-1} \rangle^{-\sigma} \sigma_k^{p-1} \right) (1 - \psi_k(x)\chi_k(\xi) ),
$$
due to the fact that the supports of $w_k^{(\alpha\beta)}(x,\xi)$ and $1-\psi_k(x)\chi_k(\xi)$ are disjoint, by Proposition \ref{compositionthm} we may write
$$
\mathtt{I}_{3,k}(x,D) \circ w_k^{(\alpha\beta)}(x,D) = R_k^{(\alpha\beta)}(x,D),
$$
where
$$
R_k^{(\alpha\beta)}(x,\xi) = N_k \int_0^1 \frac{(1-\vartheta)^{N_k-1}}{N_k!}Os-\iint e^{-iy\eta} \partialxi^{N_k} \mathtt{I}_{3,k}(x,\xi+\vartheta\eta) D_x^{N_k} w_k^{(\alpha\beta)}(x+y,\xi) dy \dslash\eta d\vartheta.
$$
The semi-norms of $R_k^{(\alpha\beta)}$ can be estimated as follows: for each $\ell_0 \in \mathbb{N}_0$, there exists $\ell_1 = \ell_1(\ell_0)$ such that
$$
| R_k |_{\ell_0}^{(0)} \leq C(\ell_0) \frac{N_k}{N_k!} | \partialxi^{N_k} \mathtt{I}_{3,k} |_{\ell_1}^{(0)} | \partialx^{N_k} w_k^{(\alpha\beta)} |_{\ell_1}^{(0)}.
$$
From Lemma \ref{lemma_estimates_w_k}, we have that
$$
| \partialx^{N_k} w_k^{(\alpha\beta)} |_{\ell_1}^{(0)} \leq C^{\ell_1 + \alpha + \beta + N_k +1} \ell_1!^{2\theta_h} (\alpha! \beta! N_k!)^{\theta_h} \sigma_k^{-N_k(p-1)}.
$$
Since $\sigma_k \to +\infty$, we can assume that $N_k \geq p$. Then
$$
\partialxi^{N_k} \mathtt{I}_{3,k} (x,\xi) = - \sum_{\stackrel{N_1 + N_2 = N_k}{N_1 \leq p-1}} \frac{N_k!}{N_1!N_2!} \partialxi^{N_1} \left( \japx^{-\sigma}\xi^{p-1} - \Theta \langle \sigma_k^{p-1} \rangle^{-\sigma} \sigma_k^{p-1} \right) \psi_k(x) \partialxi^{N_2}\chi_k(\xi).
$$
Hence,
$$
| \partialxi^{N_k} \mathtt{I}_{3,k}(x,\xi) |_{\ell_1}^{(0)} \leq C^{\ell_1+N_k+1} \ell_1!^{2\theta_h} N_k!^{\theta_h} \sigma_k^{p-1-N_k}.
$$
By gathering the previous estimates, we obtain
$$
| R_k^{(\alpha\beta)} |_{\ell_0}^{(0)} \leq C^{\alpha+\beta+N_k+1} (\alpha!\beta!)^{\theta_h} N_k!^{2\theta_h-1} \sigma_k^{p(1-N_k)-1},
$$
which provides that
$$
\| I_{3,k}(x,D) v_k^{(\alpha\beta)}(t) \|_{L^2(\R)} \leq C^{\alpha+\beta+N_k+1} (\alpha!\beta!)^{\theta_h} N_k!^{2\theta_h-1} \sigma_k^{p(1-N_k)-1}.
$$
Finally, we get
\begin{equation}\label{estimate_from_below_I_3,k}
\mathbf{Re} \left( \mathtt{I}_{3,k}(x,D) v_k^{(\alpha\beta)}(t), v_k^{(\alpha\beta)}(t) \right)_{L^2(\R)} \geq -C^{\alpha+\beta+N_k+1} (\alpha!\beta!)^{\theta_h} N_k!^{2\theta_h-1} \sigma_k^{p(1-N_k) - 1} \| v_k^{(\alpha\beta)}(t) \|_{L^2(\R)}.
\end{equation}

\end{itemize}

From \eqref{estimate_from_below_I_1,k}, \eqref{estimate_from_below_I_2,k} and \eqref{estimate_from_below_I_3,k}, we obtain that 
if the Cauchy problem \eqref{cauchy_problem_model_operator} is well-posed in $\mathcal{S}^\theta_s(\R)$, then for all $k$ sufficiently large the estimate 
\begin{eqnarray}
\mathbf{Re} \left( \japx^{-\sigma} D_x^{p-1} v_k^{(\alpha\beta)}(t), v_k^{(\alpha\beta)}(t) \right)_{L^2(\R)} & \geq & c_1 \sigma_k^{(p-1)(1-\sigma)} \| v_k^{(\alpha\beta)}(t) \|_{L^2(\R)}^2 \nonumber \\
& - & C^{\alpha+\beta+N_k+1} (\alpha!\beta!)^{\theta_h} N_k!^{2\theta_h-1} \sigma_k^{p(1-N_k)-1} \| v_k^{\alpha\beta}(t) \|_{L^2(\R)} \label{lemma_japx_D_x^p-1_below}
\end{eqnarray}
holds for some $c_1>0$ independent of $k, \alpha, \beta$ and $N_k$.
\\

\textbf{2. Estimates from above of $\| f_k^{(\alpha\beta)}(t) \|_{L^2(\R)}$.}

We know that
$$
f_k^{(\alpha\beta)} = [M,w_k^{(\alpha\beta)}]u_k = [D_t + D_x^p, w_k^{(\alpha\beta)}]u_k + [i\japx^{-\sigma}D_x^{p-1},w_k^{(\alpha\beta)}]u_k.
$$
So, our goal is to estimate the $L^2$-norms of the two commutators in the above right-hand side.

\begin{itemize}

\item \textbf{Estimate of $[D_t+D_x^p,w_k^{(\alpha\beta)}]u_k$.} Since $w_k^{(\alpha\beta)}$ does not depend on $t$, we have that
\begin{eqnarray}\label{bracket_D_t+D_x}
[D_t+D_x^p,w_k^{(\alpha\beta)}]u_k = \sum_{\gamma=1}^p \binom{p}{\gamma} D_x^\gamma w_k^{(\alpha\beta)} D_x^{p-\gamma}u_k.
\end{eqnarray}
To deal with the right-hand side of \eqref{bracket_D_t+D_x}, we will need the following identities:
\begin{equation}\label{function_derivative_of_function}
f(x) D_x^\ell g(x) = \sum_{j=0}^\ell (-1)^j \binom{\ell}{j} D_x^{\ell-j} ( g(x) D^j f(x) ),
\end{equation}
for smooth functions $f$ and $g$,
$$
\sum_{\gamma=1}^p \sum_{j=0}^{p-\gamma} c_{\gamma,j} a_{\gamma+j} = \sum_{\ell=1}^p \left( \sum_{\mathtt{q}=1}^\ell c_{\mathtt{q},\ell-\mathtt{q}} \right) a_\ell
$$
and
$$
\frac{(-1)^{\ell+1}}{\ell!} = \sum_{\mathtt{q}=1}^\ell \frac{(-1)^{\ell-\mathtt{q}}}{\mathtt{q}!(\ell-\mathtt{q})!}.
$$
Then \eqref{bracket_D_t+D_x} becomes
\begin{eqnarray}
& & [D_t + D_x^p, w_k^{(\alpha\beta)}] \nonumber \\
& = & \sum_{\gamma=1}^p \sum_{j=0}^{p-\gamma} \frac{(-1)^j p!}{\gamma! j! (p-\gamma-j)!} \left( \frac{1}{i \sigma_k^{p-1}} \right)^{\gamma+j} D_x^{p-\gamma-j} \circ w_k^{((\alpha+\ell+j)\beta)} (x,D) \nonumber \\
& = & \sum_{\ell=1}^p \left( \sum_{\mathtt{q}=1}^\ell \frac{(-1)^{\ell - \mathtt{q}}}{\mathtt{q}! (\ell-\mathtt{q})!} \right) \frac{p!}{(p-\ell)!} \left( \frac{1}{i\sigma_k^{p-1}} \right)^\ell D_x^{p-\ell} \circ w_k^{((\alpha+\ell)\beta)} (x,D) \nonumber \\
& = & \sum_{\ell=1}^p (-1)^{\ell+1} \binom{p}{\ell} \left( \frac{1}{i\sigma_k^{p-1}} \right)^\ell D_x^{p-\ell} \circ w_k^{((\alpha+\ell)\beta)}(x,D). \nonumber 
\end{eqnarray}
By Lemma \ref{lemma_derivatives_of_v_k^alphabeta}, we get
\begin{eqnarray}\label{estimate_bracket_D_t+D_x^p}
& & \left\| [ D_t + D_x^p, w_k^{(\alpha\beta)} ] u_k \right\|_{L^2(\R)} \nonumber \\
& \leq & C \sum_{\ell=1}^p \binom{p}{\ell} \frac{1}{\sigma_k^{\ell(p-1)}} \| D_x^{p-\ell} v_k^{((\alpha+\ell)\beta)} \|_{L^2(\R)} \nonumber \\
& \leq & C \sum_{\ell=1}^p \frac{1}{\sigma_k^{p(\ell-1)}} \| v_k^{((\alpha+\ell)\beta)} \|_{L^2(\R)} + C^{\alpha+\beta+N_k+1} (\alpha!\beta!)^{\theta_h} N_k!^{2\theta_h-1} \sigma_k^{p-1-N_k}.
\end{eqnarray}

\item \textbf{Estimate of $[i \japx^{-\sigma} D_x^{p-1},w_k^{(\alpha\beta)}] u_k$.} Notice that by Proposition \ref{compositionthm} of the Appendix the commutator can be written in the following way:
\begin{eqnarray}\label{identity_i_jap_D_x^p-1}
[ i \japx^{-\sigma} D_x^{p-1}, w_k^{(\alpha\beta)} ] & = & i \japx^{-\sigma} \sum_{\gamma=1}^{p-1} \binom{p-1}{\gamma} D_x^\gamma w_k^{(\alpha\beta)}(x,D) D_x^{p-1-\gamma} \nonumber \\
& - & i \sum_{\gamma=1}^{N_k-1} \frac{1}{\gamma!} D_x^\gamma \japx^{-\sigma} \partialxi^\gamma w_k^{(\alpha\beta)}(x,D) D_x^{p-1} + r_k^{(\alpha\beta)}(x,D),
\end{eqnarray}
where the symbol of $r_k^{(\alpha\beta)}(x,D)$ is
$$
r_k^{(\alpha\beta)}(x,\xi) = -iN_k \int_0^1 \frac{(1-\vartheta)^{N_k-1}}{N_k!} Os - \iint e^{-iy\eta} \partialxi^{N_k} w_k^{(\alpha\beta)}(x,\xi+\vartheta\eta) D_x^{N_k} \langle x + y \rangle^{-\sigma} \xi^{p-1} dy \dslash\eta d\vartheta.
$$
In order to estimate $r_k^{(\alpha\beta)}$, we use the support properties of $w_k^{(\alpha\beta)}$ and write
$$
\xi^{p-1} = (\xi + \vartheta\eta-\vartheta\eta)^{p-1} = \sum_{\ell=0}^{p-1} \binom{p-1}{\ell} (\xi + \vartheta\eta)^\ell (-\vartheta\eta)^{p-1-\ell}.
$$
Now we deal with the oscillatory integral by using the above identity and integration by parts, which lead us to
\begin{eqnarray}
& & Os - \iint e^{-iy\eta} \partialxi^{N_k} w_k^{(\alpha\beta)}(x,\xi+\vartheta\eta) D_x^{N_k} \langle x + y \rangle^{-\sigma} \xi^{p-1} dy \dslash\eta \nonumber \\
& = & \sum_{\ell=0}^{p-1} \binom{p-1}{\ell} \vartheta^{p-1-\ell} \nonumber \\
& \times & Os - \iint D_y^{p-1-\ell} e^{-iy\eta} \partialxi^{N_k} w_k^{(\alpha\beta)} (x,\xi+\vartheta\eta) (\xi+\vartheta\eta)^\ell D_x^{N_k} \langle x + y \rangle^{-\sigma} dy \dslash\eta \nonumber \\
& = & \sum_{\ell=0}^{p-1} \binom{p-1}{\ell} (-\vartheta)^{p-1-\ell} \nonumber \\
& \times & Os - \iint e^{-iy\eta} \partialxi^{N_k} w_k^{(\alpha\beta)}(x,\xi+\vartheta\eta) (\xi+\vartheta\eta)^\ell D_x^{N_k + p - 1 - \ell} \langle x + y \rangle^{-\sigma} dy \dslash\eta. \nonumber
\end{eqnarray}
Hence, we can estimate the semi-norms of $r_k^{(\alpha\beta)}$ in the following way:
\begin{eqnarray}
| r_k^{(\alpha\beta)} |_{\ell_0}^{(0)} & \leq & \frac{C^{N_k + p -1}}{N_k!} \sum_{\ell=0}^{p-1} | \xi^\ell \partialxi^{N_k} w_k^{(\alpha\beta)} |_{\ell_1}^{(0)} | D_x^{N_k + p - 1 - \ell} \japx^{-\sigma} |_{\ell_1}^{(0)} \nonumber \\
& \leq & C^{\alpha+\beta+N_k+1} (\alpha!\beta!)^{\theta_h} N_k!^{2\theta_h-1} \sigma_k^{p-1-N_k}.
\end{eqnarray}
By using Proposition \ref{CVthm} and the well-posedness in $\mathcal{S}^\theta_s(\R)$, it can be concluded that
\begin{equation}\label{estimate_r_k(x,D)u_k}
\| r_k^{(\alpha\beta)}(x,D) u_k \|_{L^2(\R)} \leq C^{\alpha+\beta+N_k+1} (\alpha!\beta!)^{\theta_h} N_k!^{2\theta_h-1} \sigma_k^{p-1-N_k}.
\end{equation}
For treating the term 
$$
\sum_{\gamma=1}^{N_k-1} \frac{1}{\gamma!} D_x^\gamma \japx^{-\sigma} \partialxi^\gamma w_k^{(\alpha\beta)}(x,D) D_x^{p-1},
$$ we can apply formula \eqref{function_derivative_of_function}. Then we have
\begin{eqnarray}
& & \sum_{\gamma=1}^{N_k-1} \frac{1}{\gamma!} D_x^\gamma \japx^{-\sigma} \partialxi^\gamma w_k^{(\alpha\beta)}(x,D) D_x^{p-1} \nonumber \\
& = & \sum_{\gamma=1}^{N_k-1} \sum_{\ell=0}^{p-1} \frac{1}{\gamma!} \binom{p-1}{\ell} \left( \frac{4}{\sigma_k} \right)^\gamma \left( \frac{1}{\sigma_k^{p-1}} \right)^\ell D_x^\gamma \japx^{-\sigma} D_x^{p-1-\ell} \circ w_k^{((\alpha+\ell)(\beta+\gamma))} (x,D).
\end{eqnarray}
Using the support of $D_x^{p-1-\ell} v_k^{((\alpha+\ell)(\beta+\gamma))}$ and the fact that $ |D_x^\gamma \japx^{-\sigma} | \leq C^{\gamma+1} \gamma! \japx^{-\sigma-\gamma}$, it follows that
$$
\| D_x^\gamma \japx^{-\sigma} D_x^{p-1-\ell} v_k^{((\alpha+\ell)(\beta+\gamma))} \|_{L^2(\R)} \leq C^{\gamma+1} \gamma! \sigma_k^{-(p-1)(\sigma+\gamma)} \| D_x^{p-1-\ell} v_k^{((\alpha+\ell)(\beta+\gamma))} \|_{L^2(\R)}.
$$
By Lemma \ref{lemma_derivatives_of_v_k^alphabeta}, with $N=N_k-\gamma$, we get 
\begin{eqnarray*}
 \| D_x^{p-1-\ell} v_k^{((\alpha+\ell)(\beta+\gamma))} \|_{L^2(\R)} &\leq& C \sigma_k^{p-1-\ell} \| v_k^{((\alpha+\ell)(\beta+\gamma))} \|_{L^2(\R)} \\  
&& + C^{\alpha+\ell+\beta+N_k+1} \left( \alpha! \beta! \ell! \gamma! \right)^{\theta_h} (N_k - \gamma)^{2\theta_h-1} \sigma_k^{p-1-\ell-N_k+\gamma}. 
\end{eqnarray*}
Hence,
\begin{eqnarray}
& & \| D_x^\gamma \japx^{-\sigma} D_x^{p-1-\ell} v_k^{((\alpha+\ell)(\beta+\gamma))} \|_{L^2(\R)} \leq C^{\gamma+1} \gamma! \sigma_k^{-(\sigma+\gamma)} \sigma_k^{p-1-\ell} \| v_k^{((\alpha+\ell)(\beta+\gamma))} \|_{L^2(\R)} \nonumber \\
& + & C^{\gamma+1} \gamma! \sigma_k^{-(p-1)(\sigma+\gamma)} C^{\alpha+\ell+\beta+N_k+1} ( \alpha! \beta! \ell! \gamma! )^{\theta_h} (N_k-\gamma)!^{2\theta_h-1} \sigma_k^{p-1-\ell-(N_k-\gamma)} \nonumber \\
& = & C^{\gamma+1} \gamma! \sigma_k^{-(p-1)(\sigma+\gamma)+p-1-\ell} \| v_k^{((\alpha+\ell)(\beta+\gamma))} \|_{L^2(\R)} \nonumber \\
& + & C^{\gamma+1} \gamma! \sigma_k^{-(p-1)(\sigma+\gamma)} C^{\alpha+\ell+\beta+N_k} ( \alpha! \beta! \ell! \gamma! )^{\theta_h} (N_k-\gamma)!^{2\theta_h-1} \sigma_k^{p-1-\ell-(N_k-\gamma)}. \nonumber
\end{eqnarray}
Now we can estimate
\begin{eqnarray}\label{estimate_sum_in_the_middle}
	& & \left\| \sum_{\gamma=1}^{N_k-1} \frac{1}{\gamma!} D_x^\gamma \japx^{-\sigma} \partialxi^\gamma w_k^{(\alpha\beta)}(x,D) D_x^{p-1} u_k \right\|_{L^2(\R)} \nonumber \\
	 \leq && \sum_{\gamma=1}^{N_k-1} \sum_{\ell=0}^{p-1} \frac{1}{\gamma!} \binom{p-1}{\ell} \left( \frac{4}{\sigma_k} \right)^\gamma \left( \frac{1}{\sigma_k^{p-1}} \right)^\ell  \| D_x^\gamma \japx^{-\sigma} D_x^{p-1-\ell} \circ \underbrace{w_k^{((\alpha+\ell)(\beta+\gamma))}(x,D) u_k}_{=v_k^{( ( \alpha+\ell ) ( \beta+\gamma ) )}} \|_{L^2(\R)} \nonumber \\
	 \leq && \sum_{\gamma=1}^{N_k-1} \sum_{\ell=0}^{p-1} \frac{1}{\gamma!} \binom{p-1}{\ell} 4^\gamma \sigma_k^{-\gamma} \sigma_k^{-\ell(p-1)} \left\lbrace C^{\gamma+1} \gamma! \sigma_k^{-(p-1)(\sigma+\gamma)+p-1-\ell} \| v_k^{((\alpha+\ell)(\beta+\gamma))} \|_{L^2(\R)}\right. \nonumber \\
	 && +\left. C^{\gamma+1} \gamma! \sigma_k^{-(p-1)(\sigma+\gamma)} C^{\alpha+\ell+\beta+N_k} ( \alpha! \beta! \ell! \gamma! )^{\theta_h} (N_k - \gamma)!^{2\theta_h-1} \sigma_k^{p-1-\ell-(N_k-\gamma)} \right\rbrace \nonumber \\
	 \leq && C \sum_{\ell=0}^{p-1} \sum_{\gamma=1}^{N_k-1} C^\gamma \sigma_k^{p-1-(p-1)\sigma-p(\ell+\gamma)} \| v_k^{((\alpha+\ell)(\beta+\gamma))} \|_{L^2(\R)} \nonumber \\
	  && + \hskip2pt C^{\alpha+\beta+N_k+1} (\alpha!\beta!)^{\theta_h} N_k!^{2\theta_h-1} \sigma_k^{1-N_k-(p-1)\sigma}.
\end{eqnarray}

To estimate the first term of \eqref{identity_i_jap_D_x^p-1}, we can write
\begin{eqnarray}
& & \sum_{\gamma=1}^{p-1} \binom{p-1}{\gamma} D_x^\gamma w_k^{(\alpha\beta)}(x,D) D_x^{p-1-\gamma} \nonumber \\
& = & \sum_{\ell=1}^{p-1} (-1)^{\ell+1} \binom{p-1}{\ell} \frac{1}{(i\sigma_k^{p-1})^\ell} D_x^{p-1-\ell} \circ w_k^{((\alpha+\ell)\beta)} (x,D). \nonumber
\end{eqnarray}
By using Lemma \ref{lemma_derivatives_of_v_k^alphabeta}, it follows that
\begin{multline} \label{estimate_first_term}
\left\| i \japx^{-\sigma} \sum_{\gamma=1}^{p-1} \binom{p-1}{\gamma}  D_x^\gamma w_k^{(\alpha\beta)}(x,D) D_x^{p-1-\gamma} u_k \right\|_{L^2(\R)}  \\
 =  \left\| i \japx^{-\sigma} \sum_{\ell=1}^{p-1} (-1)^{\ell+1} \binom{p-1}{\ell} \frac{1}{(i\sigma_k^{p-1})^\ell} D_x^{p-1-\ell} \circ w_k^{((\alpha+\ell)\beta)} (x,D) u_k \right\|_{L^2(\R)}  \\
 \leq  \tilde{C} \sum_{\ell=1}^{p-1} \sigma_k^{-\sigma - \ell(p-1)} \left\lbrace C \sigma_k^{p-1-\ell} \| v_k^{((\alpha+\ell)\beta))} \|_{L^2(\R)} + C^{\alpha+\ell+\beta+N_k+1} (\alpha!\beta!)^{\theta_h} N_k!^{2\theta_h-1} \sigma_k^{p-1-\ell-N_k} \right\rbrace  \\ 
  \leq  C \sum_{\ell=1}^{p-1} \sigma_k^{-(\sigma+p(\ell-1)+1)} \| v_k^{((\alpha+\ell)\beta)} \|_{L^2(\R)} + C^{\alpha+\beta+N_k+1} (\alpha!\beta!)^{\theta_h} N_k!^{2\theta_h-1} \sigma_k^{p-1-\sigma-N_k}. 
\end{multline}

From \eqref{estimate_r_k(x,D)u_k}, \eqref{estimate_sum_in_the_middle} and \eqref{estimate_first_term}, we get
\begin{eqnarray}\label{estimate_bracket_ijapx_D_x^p-1}
\left\| [ i \japx^{-\sigma} D_x^{p-1}, w_k^{(\alpha\beta)}] u_k \right\|_{L^2(\R)} & \leq & \left\| i \japx^{-\sigma} \sum_{\gamma=1}^{p-1} \binom{p-1}{\gamma} D_x^\gamma w_k^{(\alpha\beta)}(x,D) D_x^{p-1-\gamma} u_k \right\|_{L^2(\R)} \nonumber \\
& + & \left\| i \sum_{\gamma=1}^{N_k-1} \frac{1}{\gamma!} D_x^\gamma \japx^{-\sigma} \partialxi^\gamma w_k^{(\alpha\beta)}(x,D) D_x^{p-1}u_k \right\|_{L^2(\R)} \nonumber \\
& + & \| r_k^{(\alpha\beta)} (x,D) u_k \|_{L^2(\R)} \nonumber \\
& \leq & C \sum_{\ell=0}^{p-1} \sum_{\gamma=1}^{N_k-1} C^\gamma \sigma_k^{p-1-(p-1)\sigma-p(\gamma+\ell)} \| v_k^{((\alpha+\ell)(\beta+\gamma))} \|_{L^2(\R)} \nonumber \\
& + & C \sum_{\ell=1}^{p-1} \sigma_k^{p(1-\ell) - (\sigma+1)} \| v_k^{((\alpha+\ell)\beta)} \|_{L^2(\R)} \nonumber \\
& + & C^{\alpha+\beta+N_k+1} (\alpha!\beta!)^{\theta_h} N_k!^{2\theta_h-1} \sigma_k^{p-1-N_k}.
\end{eqnarray}

\end{itemize}

After all the previous computations, we obtain that if the Cauchy problem \eqref{cauchy_problem_model_operator} is well-posed in $\mathcal{S}^\theta_s(\R)$, then
\begin{eqnarray}
\| f_k^{(\alpha\beta)} \|_{L^2(\R)} & \leq & C \sum_{\ell=1}^p \frac{1}{\sigma_k^{p(\ell-1)}} \| v_k^{((\alpha+\ell)\beta)} \|_{L^2(\R)} + C \sum_{\ell=0}^{p-1} \sum_{\gamma=1}^{N_k-1} C^\gamma \sigma_k^{ p-1-(p-1)\sigma-p(\ell+\gamma) } \| v_k^{((\alpha+\ell)(\beta+\gamma))} \|_{L^2(\R)} \nonumber \\
&& + C^{\alpha+\beta+N_k+1} (\alpha!\beta!)^{\theta_h} N_k!^{2\theta_h-1} \sigma_k^{p-1-N_k}, \label{lemma_estimate_lemma_f_k_above}
\end{eqnarray}
for some positive constant $C$ which does not depend on $k,\alpha,\beta$ and $N_k$.

We are now able to conclude the proof of Proposition \ref{proposition_estimates_E_k_below}.

\begin{proof}[Proof of Proposition \ref{proposition_estimates_E_k_below}]
Inequality \eqref{norm_partial_norm_1} can be combined with estimates \eqref{lemma_japx_D_x^p-1_below} and \eqref{lemma_estimate_lemma_f_k_above} in order to obtain that
\begin{eqnarray}
\partial_t \| v_k^{(\alpha\beta)} \|_{L^2(\R)} & \geq & -\| f_k^{(\alpha\beta)} \|_{L^2(\R)} + \mathbf{Re} \left( \japx^{-\sigma} D_x^{p-1} v_k^{(\alpha\beta)},v_k^{(\alpha\beta)} \right)_{L^2(\R)} \| v_k^{(\alpha\beta)} \|_{L^2(\R)}^{-1} \nonumber \\
& \geq & c_1 \sigma_k^{(p-1)(1-\sigma)} \| v_k^{(\alpha\beta)} \|_{L^2(\R)} - C \sum_{\ell=1}^p \frac{1}{\sigma_k^{p(\ell-1)}} \| v_k^{((\alpha+\ell)\beta)} \|_{L^2(\R)} \nonumber \\
& - & C \sum_{\ell=0}^{p-1} \sum_{\gamma=1}^{N_k-1} C^\gamma \sigma_k^{ p-1-(p-1)\sigma-p(\ell+\gamma) } \| v_k^{((\alpha+\ell)(\beta+\gamma))} \|_{L^2(\R)} \nonumber \\
& - & C^{\alpha+\beta+N_k+1} (\alpha!\beta!)^{\theta_h} N_k!^{2\theta_h-1} \sigma_k^{p-1-N_k}. \nonumber
\end{eqnarray}
Then, by definition of $E_k(t)$, it follows that
\begin{eqnarray}
\partial_t E_k(t) & = & \sum_{\alpha\leq N_k,\beta\leq N_k} \frac{1}{(\alpha!\beta!)^{\theta_1}} \partial_t \| v_k^{(\alpha\beta)}(t,\cdot) \|_{L^2(\R)} \nonumber \\
& \geq & \sum_{\alpha\leq N_k,\beta\leq N_k} \frac{1}{(\alpha!\beta!)^{\theta_1}} c_1 \sigma_k^{(p-1)(1-\sigma)} \| v_k^{(\alpha\beta)} \|_{L^2(\R)} \nonumber \\
& - & C \sum_{\ell=1}^p \frac{1}{\sigma_k^{p(\ell-1)}} \sum_{\alpha\leq N_k,\beta\leq N_k} \frac{1}{(\alpha!\beta!)^{\theta_1}} \| v_k^{((\alpha+\ell)\beta)} \|_{L^2(\R)} \nonumber \\
& - & C \sum_{\ell=0}^{p-1} \sum_{\gamma=1}^{N_k-1} C^\gamma \sigma_k^{ p-1-(p-1)\sigma-p(\ell+\gamma) } \sum_{\alpha\leq N_k,\beta\leq N_k} \frac{1}{(\alpha!\beta!)^{\theta_1}} \| v_k^{((\alpha+\ell)(\beta+\gamma))} \|_{L^2(\R)} \nonumber \\
& - & \sum_{\alpha\leq N_k,\beta\leq N_k} \frac{1}{(\alpha!\beta!)^{\theta_1}} C^{\alpha+\beta+N_k+1} (\alpha!\beta!)^{\theta_h} N_k!^{2\theta_h-1} \sigma_k^{p-1-N_k}. \nonumber
\end{eqnarray}
Now our task is to treat all the terms in the right-hand side of the above inequality. Starting with the first one, we simply have
\begin{equation}\label{identity_first_term}
\sum_{\alpha\leq N_k,\beta\leq N_k} \frac{1}{(\alpha!\beta!)^{\theta_1}} c_1\sigma_k^{(p-1)(1-\sigma)} \| v_k^{(\alpha\beta)} \|_{L^2(\R)} = c_1\sigma_k^{(p-1)(1-\sigma)} E_k(t).
\end{equation}

To deal with the second term, just recall that $E_{k,\alpha+\ell,\beta}(t)=(\alpha+\ell)!^{-\theta_1}\beta!^{-\theta_1} \| v_k^{((\alpha+\ell)\beta)}(t) \|_{L^2(\R)}$. Hence,
\begin{eqnarray}
\sum_{\ell=1}^p \frac{C}{\sigma_k^{p(\ell-1)}} \sum_{\alpha\leq N_k,\beta\leq N_k} \frac{1}{(\alpha!\beta!)^{\theta_1}} \| v_k^{((\alpha+\ell)\beta)} \|_{L^2(\R)} & = & \sum_{\ell=1}^p \frac{C}{\sigma_k^{p(\ell-1)}} \sum_{\alpha\leq N_k, \beta\leq N_k} \frac{(\alpha+\ell)!^{\theta_1}}{\alpha!^{\theta_1}} E_{k,\alpha+\ell,\beta} \nonumber \\
& \leq & \sum_{\ell=1}^p \frac{C N_k^{\ell\theta_1}}{\sigma_k^{p(\ell-1)}} \left( E_k + \sum_{\alpha=N_k-\ell+1}^{N_k} E_{k,\alpha+\ell,\beta} \right). \nonumber
\end{eqnarray}
From \eqref{E_kalphabeta_estimate} in Lemma \ref{proposition_estimates_E_k_above} it holds that
$$
E_{k,\alpha+\ell,\beta} \leq C^{\alpha+\beta+\ell+1} ((\alpha+\ell)!\beta!)^{\theta_h-\theta_1},
$$
and, since $\alpha+\ell\geq N_k$ and $\theta_h-\theta_1\leq0$, we obtain that
$$
\sum_{\ell=1}^p \frac{C}{\sigma_k^{p(\ell-1)}} \sum_{\alpha\leq N_k,\beta\leq N_k} \frac{1}{(\alpha!\beta!)^{\theta_1}} \| v_k^{((\alpha+\ell)\beta)} \|_{L^2(\R)} \leq \sum_{\ell=1}^p \frac{C N_k^{\ell\theta_1}}{\sigma_k^{p(\ell-1)}} \left( E_k + C^{N_k+1} N_k!^{\theta_h-\theta_1} \right).
$$
The definition $N_k = \lfloor \sigma_k^\frac{\lambda}{\theta_1} \rfloor$ and the inequality $N_k^{N_k} \leq e^{N_k} N_k!$ allow us to estimate
\begin{eqnarray}\label{inequality_second_term}
& & \sum_{\ell=1}^p \frac{C}{\sigma_k^{p(\ell-1)}} \sum_{\alpha\leq N_k,\beta\leq N_k} \frac{1}{(\alpha!\beta!)^{\theta_1}} \| v_k^{((\alpha+\ell)\beta)} \|_{L^2(\R)} \nonumber \\
& \leq & C \sum_{\ell=1}^p \frac{(\sigma_k^\frac{\lambda}{\theta_1})^{\ell\theta_1}}{\sigma_k^{p(\ell-1)}} \left( E_k + C^{N_k+1} e^{N_k(\theta_h-\theta_1)} \sigma_k^{\frac{\lambda}{\theta_1}(\theta_h-\theta_1)N_k} \right) \nonumber \\
& \leq & C \sum_{\ell=1}^p \frac{\sigma_k^{\lambda\ell}}{\sigma_k^{p(\ell-1)}} E_k + C^{N_k+1} \sigma_k^{C - cN_k},
\end{eqnarray}
where, from now on, $c$ is a positive constant independent of $k$.

In the third term, we employ the identity $ \| v_k^{((\alpha+\ell)(\beta+\gamma))} \|_{L^2(\R)} = ( (\alpha+\ell)!(\beta+\gamma)! )^{\theta_1} E_{k,\alpha+\ell,\beta+\gamma}$, the estimate
$$
\frac{(\beta+\gamma)!}{\beta!} \leq (\beta+\gamma)^\gamma \leq (r N_k)^\gamma \leq r^\gamma (\sigma_k^\frac{\lambda}{\theta_1})^\gamma, \quad \text{provided that}\ \beta+\gamma \leq rN_k, \ r\in\mathbb{N},
$$
and the fact that, if $\lambda \in (0,1)$, then for $k$ sufficiently large it holds $C\sigma_k^{\lambda-1}<1$. Notice that
\begin{eqnarray}
& & \sum_{\ell=0}^{p-1} \sum_{\gamma=1}^{N_k-1} C^\gamma \sigma_k^{ p-1-(p-1)\sigma-p(\ell+\gamma) } \sum_{\alpha\leq N_k,\beta\leq N_k} \frac{1}{(\alpha!\beta!)^{\theta_h}} \| v_k^{((\alpha+\ell)(\beta+\gamma))} \|_{L^2(\R)} \nonumber \\
& = & \sum_{\ell=0}^{p-1} \sum_{\gamma=1}^{N_k-1} C^\gamma \sigma_k^{ p-1-(p-1)\sigma-p(\ell+\gamma) } \sum_{\alpha\leq N_k,\beta\leq N_k} \frac{((\alpha+\ell)!(\beta+\gamma)!)^{\theta_1}}{(\alpha!\beta!)^{\theta_1}} E_{k,\alpha+\ell,\beta+\gamma} \nonumber \\
& \leq & \sum_{\ell=0}^{p-1} \sum_{\gamma=1}^{N_k-1} \sum_{\alpha\leq N_k,\beta\leq N_k} \left\lbrace \sum_{\stackrel{\alpha\leq N_k-\ell}{\beta\leq N_k-\gamma}} + \sum_{\stackrel{\alpha\leq N_k,\beta\leq N_k}{\alpha+\ell>N_k \ \text{or} \ \beta+\gamma>N_k}} \right\rbrace (C\sigma_k^{\lambda-1})^\gamma \sigma_k^{ p-1-(p-1)\sigma-p(\ell+\gamma) } E_{k,\alpha+\ell,\beta+\gamma} \nonumber \\
& \leq & E_k + \sum_{\ell=0}^{p-1} \sum_{\gamma=1}^{N_k-1} \sum_{\stackrel{\alpha\leq N_k,\beta\leq N_k}{\alpha+\ell>N_k \ \text{or} \ \beta+\gamma>N_k}} C^{\alpha+\ell+\beta+\gamma+1} ( ( \alpha+\ell )! ( \beta+\gamma )! )^{\theta_h-\theta_1} \nonumber \\
& \leq & E_k + C^{N_k+1} N_k!^{\theta_h-\theta_1}. \nonumber
\end{eqnarray}
Since $N_k=\lfloor\sigma_k^\frac{\lambda}{\theta_1}\rfloor$, the above inequality turns into
\begin{equation}\label{inequality_third_term}
\sum_{\ell=0}^{p-1} \sum_{\gamma=1}^{N_k-1} C^\gamma \sigma_k^{ p-1-(p-1)\sigma-p(\ell+\gamma) } \sum_{\alpha\leq N_k,\beta\leq N_k} \frac{1}{(\alpha!\beta!)^{\theta_h}} \| v_k^{((\alpha+\ell)(\beta+\gamma))} \|_{L^2(\R)} \leq E_k + C^{N_k+1} \sigma_k^{C-cN_k},
\end{equation}
for all $k$ sufficiently large, where $C$ and $c$ are positive constants which do not depend on $k$.

For the fourth and last term, by using definition of $N_k=\lfloor \sigma_k^\frac{\lambda}{\theta_1} \rfloor$ and recalling that $\theta_h<\theta_1$, it can be concluded that
\begin{equation}\label{inequality_fourth_term}
\sum_{\alpha\leq N_k,\beta\leq N_k} \frac{1}{(\alpha!\beta!)^{\theta_1}} C^{\alpha+\beta+N_k+1} (\alpha!\beta!)^{\theta_h} N_k!^{2\theta_h-1} \sigma_k^{p-1-N_k} \leq C^{N_k+1} \sigma_k^{C-cN_k}.
\end{equation}

Therefore, gathering \eqref{identity_first_term}, \eqref{inequality_second_term}, \eqref{inequality_third_term} and \eqref{inequality_fourth_term} the proof is concluded.
\end{proof}

\begin{remark}\label{criticalcase}
In this paper we prove that the Cauchy problem \eqref{cauchy_problem_gelfand-shilov} is well-posed in $\mathcal{S}_s^\theta(\R)$ if $(p-1)\theta < \min\{\frac1{1-\sigma},s\}$ and that it is not well-posed in general if $(p-1)\theta > \min\{\frac1{1-\sigma},s\}$. The critical case $(p-1)\theta = \min\{\frac1{1-\sigma},s\}$ unfortunately remains unexplored. We first remark that in this case, the arguments of the proofs of the results of Section \ref{section_ill-posedness} fail. As a matter of fact, if $s \leq \frac1{1-\sigma}$ and $(p-1)\theta=s$, the proof of Proposition \ref{cauchy_problem_similar_schrodinger} fails since it is based on the application of Theorem 1.2 in \cite{AW24} which is not valid for $(p-1)\theta=s$. On the other hand, if $s>\frac1{1-\sigma}$ and $(p-1)(1-\sigma)=\frac1{\theta},$ we are not able to conclude from \eqref{estimate_E_k(T*)_from_below_2} that $E_{k}(T^*) \to \infty$ for all $T^* \in (0,T]$ and to prove Proposition \ref{proposition_main_result_4}. 
Concerning the validity of Theorem \ref{theorem_main_result_3} in the case $(p-1)\theta = \min \left\lbrace \frac{1}{1-\sigma}, s \right\rbrace$, it is still possible to perform the conjugation made in Section \ref{section_conjugation} but we cannot apply Theorem \ref{gevreythm}, that is Theorem 1.1 in \cite{AACM25}, since this does not cover the case $(p-1)\theta =\frac1{1-\sigma}.$ This case has been treated in \cite{ Arias_GS, KB} in the case $p=2$ but for $p \geq 3$ there are no results for the critical threshold. With a minor revision of the proof of Theorem \ref{gevreythm} it should be possible to obtain well-posedness at least in a time interval $[0,T^*]$ for some suitable $T^* <T$ but it is not clear to the authors if it is possible to extend the solution on the whole interval $[0,T]$ by a bootstrap argument. Since a revisitation of the proof of Theorem \ref{gevreythm} would require a separate analysis in Gevrey spaces, in this paper we prefer to leave this question open for a future specific investigation.  
	
\end{remark}

\begin{remark}
	Theorem \ref{theorem_main_result_3} concerns $p$-evolution equations in one space dimension. The reason is that the proof is based on the application of \cite[Theorem 1.1]{AACM25} which applies to $p$-evolution equations in one space dimension in Gevrey spaces. An extension of the latter result in higher space dimension would easily give a similar result in the Gelfand-Shilov setting as it has been obtained in \cite{Arias_GS}. At this moment the only well-posedness results in Gevrey spaces for $x \in \R^n, n \geq 2,$ concern the case $p=2$, cf. \cite{CRJEECT, KB}. In these papers, the Cauchy problem with data in Gevrey spaces is reduced to an auxiliary Cauchy problem with data in Sobolev spaces using a suitable change of variable given by a pseudodifferential operator of infinite order with symbol defined in terms of functions satisfying suitable differential inequalities. An analogous of this change of variable for $p$-evolution equations with $p \geq 3$ is still unknown and represents a challenging problem.
\end{remark}

\appendix
\begin{section}{Pseudodifferential operators}
	In this Appendix we collect some definitions and classical results about pseudodifferential operators that we use in the proof of the lemmas in Section \ref{section_ill-posedness}.
	
	\begin{definition}
			For any $m \in \R$, the set $S_{0,0}^m(\R)$ is defined as the space of all functions $p \in C^\infty(\R^2)$ satisfying the following conditions: for all $\alpha,\beta \in \mathbb{N}_0$, there exists a constant $C_{\alpha,\beta}>0$ such that
			$$
			| \doublepartial p(x,\xi) | \leq C_{\alpha,\beta} \japxi^m.
			$$
			The topology in $S_{0,0}^m(\R)$ is induced by the family of semi-norms
			$$
			| p |_\ell^{(m)} := \max_{\alpha\leq\ell,\beta\leq\ell} \sup_{x,\xi\in\R} | \doublepartial p(x,\xi) | \japxi^{-m}, \quad p \in S_{0,0}^m(\R), \ \ell \in \mathbb{N}_0.
			$$
	\end{definition}
 		
	\begin{proposition}[Theorem 1.6 in \cite{kumano}] \label{CVthm}
		If $p \in S_{0,0}^m(\R)$, then for all $s \in \R$ there exist $\ell:=\ell(s,m) \in \mathbb{N}_0$ and $C:=C_{s,m}>0$ such that
		$$
		\| p(x,D)u \|_{H^s(\R)} \leq C | p |_\ell^{(m)} \| u \|_{H^{s+m}(\R)}, \quad \forall u \in H^{s+m}(\R).
		$$
		Besides, when $m=s=0$, $| p |_\ell^{(m)}$ can be replaced by
		$$
		\max_{\alpha,\beta\leq2} \sup_{x,\xi\in\R} | \doublepartial p(x,\xi) |.
		$$
	\end{proposition}

	Now we consider the algebra properties of $S_{0,0}^m(\R)$ with respect to the composition of operators. Let $p_j \in S_{0,0}^{m_j}(\R)$, $j=1,2$, and define
	\begin{eqnarray}\label{symbol_composition_S_0,0^m}
		q(x,\xi) & = & Os - \iint e^{-iy\eta} p_1(x,\xi+\eta) p_2(x+y,\xi) dy \dslash\eta \nonumber \\
		& = & \lim_{\epsilon\to0} \iint e^{-iy\eta} p_1(x,\xi+\eta) p_2(x+y,\xi) e^{-\epsilon^2 y^2} e^{-\epsilon^2 \eta^2} dy \dslash\eta.
	\end{eqnarray} We have the following result concerning the composition, whose proof can be checked in \cite{kumano}, Lemma 2.4 and Theorem 1.4.

	\begin{proposition}\label{compositionthm}
		Let $p_j \in S_{0,0}^{m_j}(\R)$, $j=1,2$, and $q$ defined by \eqref{symbol_composition_S_0,0^m}. Then $q \in S_{0,0}^{m_1+m_2}(\R)$ and $q(x,D)=p_1(x,D)p_2(x,D)$. Moreover,
		\begin{equation}\label{asymptotic_q_composition}
			q(x,\xi) = \sum_{\alpha<N} \frac{1}{\alpha!} \partialxi^\alpha p_1(x,\xi) D_x^\alpha p_2(x,\xi) + r_N(x,\xi),
		\end{equation}
		where
		$$
		r_N(x,\xi) = N \int_0^1 \frac{(1-\vartheta)^{N-1}}{N!} Os-\iint e^{-iy\eta}\partialxi^N p_1(x,\xi+\vartheta\eta) D_x^N p_2(x+y,\xi) dy \dslash\eta d\vartheta,
		$$
		and the semi-norms of $r_N$ can be estimated in the following way: for any $\ell_0 \in \mathbb{N}_0$, there exists $\ell_1=\ell_1(\ell_0)$ such that
		$$
		| r_N |_{\ell_0}^{(m_1+m_2)} \leq C_{\ell_0} | \partialxi^N p_1 |_{\ell_1}^{(m_1)} | \partialx^N p_2 |_{\ell_1}^{(m_2)}.
		$$
	\end{proposition}
In the paper we shall also deal with symbols of $\textrm{\textbf{SG}}$ type. We recall here the definition.

\begin{definition}\label{SGsymbols}
	Fixed $m_1,m_2 \in \R$, we denote by $\mathbf{SG}^{m_1,m_2} (\R^2)$ the space of all functions $p \in C^\infty(\R^2)$ satisfying for all $\alpha,\beta \in \mathbb{N}_0$, the following estimate
	$$
	| \doublepartial p(x,\xi) |\japxi^{-m_1+\alpha} \japx^{-m_2+\beta} <\infty.
	$$
\end{definition}
	
	The next result is a variant of sharp G{\aa}rding inequality for $\mathbf{SG}$ symbols which follows directly from Theorem 4 and Proposition 6 in \cite{ACsharpGarding}.

	\begin{proposition}\label{theorem_sharp_garding}
		Let $p \in \mathbf{SG}^{m_1,m_2}(\R^{2n})$ such that $\mathbf{Re} \ p(x,\xi) \geq 0$. Then there exist $q \in \mathbf{SG}^{m_1,m_2}(\R^{2n})$ and $r \in \mathbf{SG}^{m_1-1,m_2-1}(\R^{2n})$ such that
		$$
		p(x,D) = q(x,D) + r(x,D) 	$$
		and
		$$
		\left( q(x,D)v, v \right)_{L^2(\R^{n})} \geq 0, \quad \forall v \in \mathcal{S}(\R^{n}).
		$$
	\end{proposition}

	\end{section}

\textbf{Acknowledgements.} The authors wish to thank the referees for their precious suggestions which allowed them to improve the quality of the paper.



\end{document}